\documentclass[12pt]{amsart}
\usepackage{mdframed}
\usepackage{amsfonts,amscd,amssymb,amsmath,amsthm}
\usepackage{etex}
\usepackage{caption,graphicx,mathrsfs,stmaryrd}
\usepackage{geometry}
\usepackage{marginnote}
\usepackage{verbatim}
\usepackage{mathtools}
\usepackage{enumerate}
\usepackage{esint}

\usepackage{wasysym}
\usepackage{nccmath}

\usepackage[english]{babel}
\usepackage[utf8]{inputenc}

\usepackage{xparse}

\usepackage{wrapfig,subcaption}
\usepackage[colorlinks,citecolor=blue]{hyperref}

\usepackage[
    backend=biber,
    style=alphabetic,
    sortlocale=auto,
    natbib=true,
    url=false, 
    doi=true,
    eprint=false,
    safeinputenc=true,
    giveninits=true,
    isbn=false,
    maxbibnames=10
]{biblatex}
\definecolor{Maroon}{rgb}{1.0,0.2,0.5}

\def\polhk#1{\setbox0=\hbox{#1}{\ooalign{\hidewidth
    \lower1.5ex\hbox{`}\hidewidth\crcr\unhbox0}}} 

\newtheorem{theorem}{Theorem}[section]
\newtheorem{lemma}[theorem]{Lemma}
\newtheorem{corollary}[theorem]{Corollary}

\newtheorem{proposition}[theorem]{Proposition}
\theoremstyle{definition}

\newtheorem{claim}[theorem]{Claim}
\newtheorem{remark}[theorem]{Remark}
\newtheorem{exttheorem}{Theorem}

\newcommand{\Def}{\coloneqq}

\newcommand{\R}{\mathbb{R}}
\newcommand{\C}{\mathbb{C}}
\newcommand{\D}{\mathbb{D}}
\newcommand{\A}{\mathcal{A}}
\newcommand{\F}{\mathcal{F}}
\newcommand{\Q}{\mathbb{Q}}
\newcommand{\N}{\mathbb{N}} 

\newcommand{\Z}{\mathbb{Z}}
\newcommand{\T}{\mathbb{T}}
\newcommand{\U}{\mathcal{U}}
\newcommand{\LP}{\operatorname{L}}

\newcommand{\Matrices}[1]{ \mathbb{M}\left( {#1} \right) } 
\newcommand{\SAMatrices}[1]{ \mathbb{M}_{SA}\left( {#1} \right) } 
\DeclareDocumentCommand{\Matrices}{ O{r} O{n} }{ \mathbb{M}_{#2} \left( {#1} \right) }
\DeclareDocumentCommand{\SAMatrices}{ O{r} O{n} }{ \mathbb{M}_{#2,SA} \left( {#1} \right) }

\newcommand{\Id}{\operatorname{Id}}

\newcommand{\stdcg}{N_\bbc(0,1)}

\newcommand{\Exp}{\mathbb{E}}
\newcommand{\E}{\mathbb{E}}
\newcommand{\prob}{\mathbb{P}}

\renewcommand{\Pr}{\prob}
\DeclareDocumentCommand \one { o }
{%
\IfNoValueTF {#1}
{\mathbf{1}  }
{\mathbf{1}\left\{ {#1} \right\} }%
}

\newcommand{\Var}{\operatorname{Var}}

\newcommand{\As}{\ensuremath{\operatorname{a.s.}}}

\DeclareDocumentCommand{\Prto} {o} {
\IfNoValueTF {#1}
 {\overset{\Pr}{\longrightarrow}}
 { \xrightarrow[ #1 \to \infty]{\Pr }}
}
\DeclareDocumentCommand{\Asto} {o} {
\IfNoValueTF {#1}
 {\overset{\operatorname{a.s.}}{\longrightarrow}}
 {
 \xrightarrow[ #1 \to \infty]{\operatorname{a.s.} }
 }
}
\DeclareDocumentCommand{\Mgfto} {o} {
\IfNoValueTF {#1}
{\overset{\operatorname{mgf}}{\longrightarrow}}
{ \xrightarrow[ #1 \to \infty]{\operatorname{mgf} }}
}

\DeclareDocumentCommand{\Wkto} {o} {
\IfNoValueTF {#1}
 {\overset{(d)}{\longrightarrow}}
 { \xrightarrow[ #1 \to \infty]{(d) }}
}

\DeclareDocumentCommand \LPto { O{1} }
{\overset{\operatorname{\LP^{#1}}}{\longrightarrow}}

\newcommand{\bbc}{\mathbb{C}}

\newcommand{\bbn}{\mathbb{N}}

\newcommand{\dd}{\mathrm{d}}

\newcommand{\VMOA}{\ensuremath{\operatorname{VMOA} }}
\newcommand{\BMOA}{\ensuremath{\operatorname{BMOA} }}
\newcommand{\Sledd}{\ensuremath{\operatorname{SL} }}

\newcommand{\Bloch}{\ensuremath{\mathcal{B} }}

\DeclareDocumentCommand{\starnorm} {o} {
\IfNoValueTF {#1}
{ \left\| \cdot \right\|_{*} }
{ \left\| { #1 } \right\|_{*} }
}

\DeclareDocumentCommand{\Blochnorm} {o} {
\IfNoValueTF {#1}
{ \left\| \cdot \right\|_{\mathcal{B}} }
{ \left\| { #1 } \right\|_{\mathcal{B}} }
}

\DeclareDocumentCommand{\RSledd} {O{\cdot} d<> } {
\IfNoValueTF {#2}
{ \left\| { #1 } \right\|_{S(R)} }
{ \left\| { #1 } \right\|_{S(R),{#2}} }
}

\DeclareDocumentCommand{\TSledd} {O{\cdot} d<>} {
\IfNoValueTF {#2}
{ \left\| { #1 } \right\|_{S(T)} }
{ \left\| { #1 } \right\|_{S(T),#2} }
}

\begin{document}

\title[GAFs of BMO]{Gaussian analytic functions of\\
Bounded Mean oscillation}
\author{Alon Nishry}
\address{School of Mathematics, Department of Pure Mathematics, Tel Aviv University }
\email{alonish@tauex.tau.ac.il}

\author{Elliot Paquette}
\address{Department of Mathematics, The Ohio State University}
\email{paquette.30@osu.edu}
\thanks{
  The work of both authors was supported by grants
  from the Israel Science Foundation, from the US-Israel Binational 
  Science Foundation, and from the European Research Council.
}
\date{\today}
\maketitle
\begin{abstract}
We consider random analytic functions given by a Taylor series with independent, centered complex Gaussian coefficients. We give a new sufficient condition for such a function to have bounded mean oscillations. Under a mild regularity assumption this condition is optimal. Using a theorem of Holland and Walsh, we give as a corollary a new bound for the norm of a random Gaussian Hankel matrix. Finally, we construct some \emph{exceptional} Gaussian analytic functions which in particular disprove the conjecture that a random analytic function with bounded mean oscillations always has \emph{vanishing} mean oscillations.
\end{abstract}

\section{Introduction}\label{sec:intro}
Functions with random Fourier (or Taylor) coefficients play an important role in harmonic and complex analysis, e.g. in the proof of de Leeuw, Kahane, and Katznelson \cite{deLeeuw1977} that Fourier coefficients of continuous functions can majorize any sequence in $\ell^2$. A well known phenomena is that series with independent random coefficients are much `nicer' than an arbitrary function would be. For example, a theorem of Paley and Zygmund \cite[Chapter 5, Proposition 10]{Kahane, PaleyZygmund} states that a Fourier series with square summable coefficients and random signs almost surely represents a \emph{subgaussian} function on the circle.

In this paper we choose to focus on one particularly nice model of random analytic functions, the \emph{Gaussian analytic functions} (GAFs). A GAF is given by a random Taylor series
\begin{equation}\label{eq:GAF}%
G\left(z\right)=\sum_{n=0}^{\infty}a_{n}\xi_{n}z^{n},
\end{equation}
where $\left\{ \xi_{n}\right\} _{n\ge 0}$ is
a sequence of independent standard complex Gaussians (i.e.\ with density $\frac{1}{\pi} e^{-|z|^2}$ with respect to Lebesgue measure on the complex plane $\bbc$)
and where $\left\{ a_{n}\right\}_{n\ge 0}$ is a sequence of 
non--negative constants. Many of the results we cite can be extended to more general probability distributions, and it is likely that our results can be similarly generalized, but we will not pursue this here. For recent accounts of random Taylor series, many of which focus on the distributions of their zeros, see for example \cite{nazarov2010random,ZerosBook}. A classical book on this and related subjects is \cite{Kahane}.

We are interested in properties of the sequence $\left\{ a_n \right\}$ that imply various regularity and finiteness properties of the function $G$ represented by the series \eqref{eq:GAF}. An important early effort is the aforementioned paper \cite{PaleyZygmund}, in which it was established that $G$ is almost surely in $\cap_{0 < p < \infty} \text{H}^p$ if and only if $\left\{ a_n \right\} \in \ell^2.$ We recall that $\text{H}^p$ can be characterized as the space of analytic functions whose \emph{non--tangential} boundary values on $\T= \left\{ z : |z|=1 \right\}$ exist and are in $\text{L}^p.$ One should compare this result with the well-known fact that a non-random analytic function belongs to $H^2$ if and only if the sequence of its Taylor coefficients is square summable. The related question of when $G$ is almost surely in $\text{H}^\infty,$ the \emph{bounded} analytic functions on the unit disk is substantially more involved (see \cite{MarcusPisier}).

To fix ideas, let us make for a moment a few simplifying assumptions about the coefficients $\{a_n\}$ of the series \eqref{eq:GAF}. We assume $a_0 = 0$,
and denote by
$$
\sigma_k^2 = \sum_{n=2^{k}}^{2^{k+1}-1} a_n^2,\qquad k\in \{0, 1,2,\dots\},
$$
the total variance of the dyadic blocks of coefficients. We say that the sequence $\{a_n\}$ (or equivalently $G$) is \emph{dyadic-regular} if the sequence $\{\sigma_k\}$ is \emph{decreasing} as $k \to \infty$. It is known (see \cite[Chapters 7 and 8]{Kahane}) that if $G$ is dyadic-regular, then $G$ is almost surely in $H^\infty$ if and only if
\begin{equation}\label{eq:bounded_cond}
\sum_{k=0}^\infty \sigma_k < \infty, \qquad i.e. \quad \{\sigma_k\} \in \ell^1.
\end{equation}
Moreover, if the series in \eqref{eq:bounded_cond} converges, then $G$ is almost surely \emph{continuous}. Hence, a bounded random series \emph{gains} additional regularity. 

For a space $S$ of analytic functions on the unit disk, let $S_G$ be the set of coefficients $\{a_n\}$ for which a GAF $G \in S$ almost surely. If $S \subsetneq T$, and $S_G = T_G$, then we say that GAFs have a \emph{regularity boost} from $T$ to $S$, e.g. $C_G = H^\infty_G$. This regularity boost can be viewed as a manifestation of a general probabilistic principle: a Borel probability measure on a complete metric space tends to be concentrated on a \emph{separable} subset of that space.\footnote{Under the continuum hypothesis, by the main theorem of \cite{MarczewskiSikorski}, any Borel probability measure on a metric space with the cardinality of the continuum is supported on a separable subset.}

Clearly there is a gap between \eqref{eq:bounded_cond} and the Paley-Zygmund condition $\{\sigma_k\} \in \ell^2$. A well-known function space that lies \emph{strictly} between $H^\infty$ and $\cap_{0 < p < \infty} \text{H}^p$ is the space of analytic functions of \emph{bounded mean oscillation} or $\BMOA$ (e.g. see \cite[Equation (5.4)]{Girela}). For an interval $I \subseteq [0,1],$ and any $f \in \text{L}^1(\T),$ put 

\begin{equation}
  M_{I}(f) \coloneqq \fint_{I}  | f(e(\theta)) - \textstyle{\fint_I }f|\,\dd \theta,
  \qquad
  \text{where}
  \quad
  \fint_I f \coloneqq
  \frac{1}{|I|} \int_{I} f(x)\,\dd x.
  \label{eq:MI}
\end{equation}
Define the seminorm on $\text{H}^1,$
\begin{equation}
  \starnorm[F]
  =\sup_{I \subseteq [0,1]} M_I(F).
  \label{eq:starnorm}
\end{equation}
The restriction for $F \in \text{H}^1$ is necessary for $F$ to have non--tangential boundary values in $\text{L}^1$ on the unit disk.  On the subspace of $\text{H}^1$ in which $F(0) = 0,$ this becomes a norm.  We may take $\BMOA$ to be the (closed) subspace of $\text{H}^1$ for which $\starnorm$ is finite.%

The space $\BMOA$ is famously \cite{FeffermanStein} the dual space of $\text{H}^1$ with respect to the bilinear form on analytic functions of the unit disk given by
\[
  (F,G) = \lim_{r\to1} \int_0^1 F(re(\theta))\overline{ G\left( re(\theta) \right)}\,\dd \theta,
\]
and in many aspects it serves as a convenient `replacement' for the space $H^\infty$. However, $\BMOA$ is not separable (\cite[Corollary 5.4]{Girela}).

One of our main results is the following.
\begin{theorem} \label{thm:dy_reg_in_VMOA}
A dyadic-regular Gaussian analytic function $G$ that satisfies the Paley-Zygmund condition $\{\sigma_k\} \in \ell^2$ almost surely belongs to $\VMOA$ - the space of analytic functions of \emph{vanishing} mean oscillation.
\end{theorem}

The space $\VMOA$ is the closure of polynomials (or continuous functions) in the norm $\starnorm$, and hence it is separable. It can alternatively be characterized as the subspace of $\text{H}^1$ for which $\lim_{|I|\to 0} M_I^1(F) = 0.$ In fact, we show that a dyadic-regular GAF with square-summable coefficients almost surely belongs to a subspace of $\VMOA$, which we attribute to Sledd \cite{sledd1981random}.

\subsection{The Sledd Space $\Sledd$}

Sledd \cite{sledd1981random} introduced a function space, which is  contained in $\BMOA$ and much more amenable to analysis. Define the seminorm for $F \subset H^1$
\begin{equation}\label{eq:TSleddintro}
  \TSledd[F]^2 = \sup_{|x|=1} \sum_{n=0}^\infty | T_n \star F(x)|^2,
\end{equation}
where $\star$ denotes convolution, and $\{T_n\}$ is a certain sequence of compactly supported bump functions in Fourier space, so that $\widehat{T_n}=1$ for modes from $[2^n, 2^{n+1}]$  (see \eqref{eq:trapezoid} for details).  We let $\Sledd$ denote the subspace of $\text{H}^1$ with finite $\TSledd$ norm; \cite{sledd1981random} shows that $\Sledd \subsetneq \BMOA$.\footnote{The function $I_F= \sum_{n=0}^\infty | T_n \star F(x)|^2$ is essentially what appears in Littlewood--Paley theory.  For each $\frac23 < p < \infty,$ finiteness of the $p$--norm of $I_F$ is equivalent to being in $\text{H}^p,$ \cite[Theorem 5]{Stein}.  Thus, in some sense $\Sledd$ could be viewed as a natural point in the hierarchy of $\text{H}^{p}$ spaces.}
Sledd proved the following result.
\begin{exttheorem}[{\cite[Theorem 3.2]{sledd1981random}}] \label{thm:sleddcoefs}
  If $\{\sqrt{k} \, \sigma_k \} \in \ell^2$, then $G \in \VMOA$ almost surely.
\end{exttheorem}
\begin{remark}
Sledd proved the result for series with random signs, but his method works also in our setting. In fact his theorem shows that $G$ is almost surely in $\VMOA \cap \Sledd.$
\end{remark}

We extend the analysis of the $\TSledd$ seminorm, and in particular find a better sufficient condition for the finiteness of $\TSledd[G]$. %
\begin{theorem}\label{thm:bettersledd}
  If $\sum_{k=1}^\infty \sup_{n \geq k}\{ \sigma_n^2 \} < \infty,$ then $G \in \Sledd$ almost surely.
\end{theorem}
In particular, if $G$ is dyadic-regular and $\{ \sigma_k \} \in \ell^2$, then $G \in \Sledd.$  The latter condition is necessary for $G$ to have well--defined boundary values, and so we see that under the monotonicity assumption, a GAF $G$ which has boundary values in $\text{L}^2$ is in $\BMOA.$  We also note that the condition in Theorem \ref{thm:bettersledd} is strictly weaker than the one in Theorem \ref{thm:sleddcoefs} (see Lemma \ref{lem:lac}).

The Sledd space $\Sledd$ is non--separable (see Proposition \ref{prop:sleddnonsep}). 
The proof of Theorem \ref{thm:sleddcoefs} is based on a stronger condition than $\TSledd[G] {}< \infty$, that in addition implies that a function is in the space $\Sledd \cap \VMOA$.\footnote{ Specifically, \cite{sledd1981random} shows that under the condition in Theorem \ref{thm:sleddcoefs}, $\sum_{n=0}^\infty \sup_{|x|=1}  | T_n \star F(x)|^2$ is finite, which implies $F \in \Sledd \cap \VMOA.$}
We show that this is unnecessary, as a GAF which is in $\Sledd$ has a regularity boost:
\begin{theorem}\label{thm:sleddVMOA}
  If $G \in \Sledd$ almost surely, then $G \in \VMOA$ almost surely.
\end{theorem}

This could raise suspicion that there is also a regularity boost from $\BMOA$ to $\VMOA,$ which is perhaps the most natural separable subspace of $\BMOA$.  Indeed, Sledd \cite{sledd1981random} asks whether it is possible to construct a non-$\VMOA$ random analytic function in $\BMOA$.

\subsection{Exceptional Gaussian analytic functions}
Sledd \cite[Theorem 3.5]{sledd1981random} gives a construction of a random analytic function with square summable coefficients which is not in $\BMOA$, and moreover is not \emph{Bloch} (this construction can be easily adapted to GAFs). The Bloch space $\Bloch,$ contains all analytic functions $F$ on the unit disk for which
\begin{equation}
  \Blochnorm[F] \Def \sup_{|z| \leq 1} (  (1-|z|^2)|F'(z)|) < \infty.
  \label{eq:bloch}
\end{equation}
See \cite{Girela} or \cite{anderson1974bloch} for more background on this space. Gao \cite{gao2000characterization} provides a complete characterization of which sequences of coefficients $\left\{ a_n \right\}$ give GAFs in $\Bloch.$

The space $\Bloch$ is non--separable, suggesting that GAFs in $\Bloch$ could concentrate on much smaller space. Finding this space is a natural open question and does not seem obvious from the characterization in \cite{gao2000characterization}. It is known that $\BMOA \subset \Bloch$ (e.g. \cite[Corollary 5.2]{Girela}), and \emph{a priori}, it could  be that GAFs which are in $\text{H}^2 \cap \Bloch$ are automatically in $\BMOA.$ However, our following result disproves this, and also answers the aforementioned question of Sledd.

\begin{theorem}\label{thm:strict}
We have,
\begin{equation}\label{eq:strict}
  \Sledd_G
  \subsetneq
  \VMOA_G  
  \subsetneq
  \BMOA_G  
  \subsetneq
  (\text{H}^2 \cap \Bloch)_G.
\end{equation}
\end{theorem}

\begin{remark}
From Theorem \ref{thm:bettersledd}, and standard results on boundedness of Gaussian processes, we may add that $\text{H}^{\infty}_G \subsetneq \Sledd_G.$  From the example in \cite{sledd1981random}, it also follows that $(\text{H}^2 \cap \Bloch)_G \subsetneq \text{H}^2_G$ 
\end{remark}

We leave open the question of the existence of a natural separable subspace $S$ of $\BMOA$ such that $\BMOA_G=S_G.$ 

\subsection{Some previously known results}
Billard \cite{Billard} (see also \cite[Chapter 5]{Kahane}) proved that a random analytic function with independent symmetric coefficients is almost surely in $\text{H}^\infty$ if and only if it almost surely extends continuously to the closed unit disk.

A complete characterization of Gaussian analytic functions which are bounded on the unit disk was found by Marcus and Pisier \cite{MarcusPisier} in terms of rearrangements of the covariance function (see also \cite[Chapter 15]{Kahane}). They moreover show the answer is the same for Steinhaus and Rademacher random series (where the common law of all $\{\xi_n\}$ is taken uniform on the unit circle or on $\left\{ \pm 1 \right\},$ respectively).  Their criterion can be seen to be equivalent to the finiteness of Dudley's entropy integral for the process of boundary values of $G$ on the unit circle.

The best existing sufficient conditions that we know for the sequence $\left\{ a_n \right\}$ to belong to $\BMOA_G$ are due to \cite{sledd1981random} and \cite{Hasi}. The conditions of \cite{Hasi} are complicated to state, but they recover the condition in Theorem \ref{thm:sleddcoefs}. If we define
\begin{equation}\label{eq:sigma}
  B_k^2 = \sum_{n=1}^k 2^{2n}\sigma_n^2
  \qquad
  \text{for}\quad
  k \in \{0,1,2,\dots\},
\end{equation}
then one corollary of \cite{Hasi} is:
\begin{exttheorem} \label{thm:Hasi}
  If for some $0 < p \leq 2,$ 
  \[
  \sum_{n=1}^\infty 2^{-2pn} \biggl(\sum_{k=1}^n 2^k\sqrt{k}B_k\biggr)^p < \infty,
  \]
  then $G \in \BMOA$ almost surely.  
\end{exttheorem}

\subsection{Norms of random Hankel matrices}

A \emph{Hankel} matrix $A$ is any $n \times n$ matrix with the structure $A_{ij} = (c_{i+j-2})$ for some sequence $\left\{ c_i \right\}_0^\infty.$  We will consider the case that $n \in \N$, and we will consider the infinite case, in which case we refer to $A$ as the Hankel operator on $\ell^2,$ which may well be unbounded.
Let $\phi(z) = \sum_{n=0}^\infty c_n z^{n+1}$ be the \emph{symbol} of $A$.  Then from a Theorem of \cite[Part III]{Holland}, there is an absolute constant $M$ so that
\begin{equation}\label{eq:HW}
  \frac{1}{M} \|\phi\|_* \leq \|A\| \leq M\|\phi\|_*,
\end{equation}
with $\|A\|$ the operator norm of $A.$

If we take $c_m = a_{m+1} \xi_{m+1}$ for all $m \geq 0$ with $\{\xi_m\}$ i.i.d.\ $\stdcg$  and with $a_m \geq 0$ for all $m,$ then $\phi$ is exactly the GAF $G$.  Moreover by combining Theorem \ref{thm:sledd}, Remark \ref{rem:sleddequiv} and Lemma \ref{lem:g2sucks}, we have that there is an absolute constant $C>0$ so that
\[
  \Exp \|\phi\|_*^2 \leq C\sum_{k=1}^\infty \sup_{m \geq k}\{ \sigma_m^2 \}.
\]

Note that for any $n \times n$ Hankel matrix $A$ with symbol $\phi(z)=\sum_{k=0}^\infty c_k z^{k+1},$ if $B$ is the infinite Hankel operator with symbol $\phi_n(z) = \sum_{k=0}^{2n} c_k z^{k+1},$ then $\|A\| \leq \|B\|.$We arrive at the following corollary:
\begin{theorem}\label{thm:hankel}
  There is an absolute constant $C>0$ so that
  if $A$ is an $n \times n$ Hankel matrix with symbol $G$ (i.e.\ a centered complex Gaussian Hankel matrix), with $L$ the smallest integer greater than or equal to $\log_2(2n),$
  \[
    \Exp \|A\|^2
    \leq
    C
    \sum_{k=0}^{L}
    \sup_{k \leq m \leq L}
    \sigma_m^2.
  \]
\end{theorem}
\noindent We emphasize that by virtue of \eqref{eq:HW} the problem of estimating the norm of a random Gaussian Hankel matrix is essentially equivalent to the problem of estimating the $\starnorm$ norm of a GAF.

This is particularly relevant as random Hankel and Toeplitz matrices\footnote{A Toeplitz matrix $A$ has the form $A_{ij} = w_{i-j}$ for some $(w_k)_{-\infty}^{\infty}.$  The symbol for such a matrix is again $\sum w_k z^k.$  
By reordering the rows, it can be seen that a Toeplitz matrix with symbol $\sum_{-n}^n w_kz^k$ has the same norm as the Hankel matrix with symbol $\sum_{0}^{2n} w_{k-n} z^k.$} have appeared many times in the literature and have numerous applications to various statistical problems.  See the discussion in \cite{Bryc} for some details.  The particular case where $G$ is a \emph{real} Kac polynomial, so that the antidiagonals of $A$ are i.i.d.\ Gaussian, is particularly well studied.  In that case, \cite{Meckes}, \cite{Adamczak} and \cite{Nekrutkin} give proofs that $\Exp \|A\| \leq c\sqrt{n \log n}$ (finer results for the symmetric Toeplitz case are available in \cite{Sen}).  Note that Theorem \ref{thm:hankel} gives the same order growth for an $n \times n$ Hankel matrix with independent standard \emph{complex} normal antidiagonals.  

Furthermore, \cite{Meckes} gives a matching lower bound, and his method can be applied to show that (deterministically) 
\[
  \|A\| \geq \sup_{|z|=1} 
  \biggl|
  \sum_{j=0}^{2(n-1)} (1-\tfrac{|n-1-j|}{n})a_j \xi_j z^j
  \biggr|.
\]
Fernique's theorem \cite[Chapter 15, Theorem 5]{Kahane} can then be used to show that Theorem \ref{thm:hankel} is sharp up to multiplicative numerical constant, at least when $a_j = j^{-\alpha}$ for $\alpha \in \R.$

\subsection*{Organization}

In Section \ref{sec:preliminaries}, we give some background theory for working with GAFs and random series.
In Section \ref{sec:sledd}, we give some further properties of the space $\Sledd$ and we give some equivalent characterizations for $G \in \Sledd.$  We also prove Theorem \ref{thm:sleddVMOA}.
In Section \ref{sec:sufficiency}, we give a sufficient condition for $G$ to be in $\Sledd$; in particular we prove Theorem \ref{thm:bettersledd}.
Finally in Section \ref{sec:exceptional}, we construct exceptional GAFs, and we show the inclusions in \eqref{eq:strict} are strict.

\subsection*{Notation}

We use the expression \emph{numerical constant} and \emph{absolute constant} to refer to fixed real numbers without dependence on any parameters.
We make use of the notation $\lesssim$ and $\gtrsim$ and $\asymp.$  In particular we say that $f(a,b,c,\dots) \lesssim g(a,b,c,\dots)$ if there is an absolute constant $C>0$ so that $f(a,b,c,\dots) \leq Cg(a,b,c,\dots)$ for all $a,b,c,\dots.$  We use $f \asymp g$ to mean $f \lesssim g$ and $f \gtrsim g.$

\subsection*{Acknowledgements}

We are very thankful to Gady Kozma, to whom the proof of Lemma \ref{lem:gady} is due. Most of the work was performed during mutual visits of the authors, which were supported by a grant from the United States - Israel Binational Science Foundation (BSF Start up Grant no. 2018341). A. N. was supported in part by ISF Grant 1903/18.

\section{Preliminaries}
\label{sec:preliminaries}

Some of our proofs will rely on the so--called contraction principle.
\begin{proposition}[Contraction Principle]
  \label{prop:contraction}
  For any finite sequence $\left( x_i \right)$ in a topological vector space $V,$ any continuous convex $F : V \to [0,\infty]$, any i.i.d., symmetrically distributed random variables $\left( \epsilon_i \right)$ and any $\left( \alpha_i \right)$ real numbers in $[-1,1]$:
  \begin{enumerate}[\textnormal{(}i\textnormal{)}]
    \item
      \(
      \Exp F \left( \sum_{i} \alpha_i \epsilon_i x_i \right)
      \leq
      \Exp F \left( \sum_{i} \epsilon_i x_i \right).
      \)
    \item If $F$ is a seminorm, then for all $t > 0,$
      \(
      \Pr\left[  F \left( \sum_{i} \alpha_i \epsilon_i x_i \right) \geq t \right]
      \leq 2\Pr\left[  F \left( \sum_{i} \epsilon_i x_i \right) \geq t \right].
      \)
  \end{enumerate}
\end{proposition}
This is essentially \cite[Theorem 4.4]{LedouxTalagrand}, although we have changed the formulation slightly.  For convenenience we sketch the proof.
\begin{proof}
  The mapping
  \[
    (\alpha_1, \alpha_2, \dots, \alpha_N) \mapsto \Exp F \left( \sum_{i} \alpha_i \epsilon_i x_i \right)
  \]
  is convex.  Therefore it attains its maximum on $[-1,1]^N$ at an extreme point, i.e.\ an element of $\left\{ \pm 1 \right\}^N.$  By the symmetry of the distributions, for all such extreme points, the value of the expectation is
  \(
      \Exp F \left( \sum_{i} \epsilon_i x_i \right),
  \)
  which completes the proof of the first part.

  For the second part, we may without loss of generality assume that $\alpha_1 \geq \alpha_2 \geq \dots \geq \alpha_N \geq \alpha_{N+1} = 0$ by relabeling the variables and using the symmetry of the distributions of $\left\{ \epsilon_i \right\}.$  Letting $S_n =  \sum_{k=1}^n \epsilon_i x_i$ for any $1 \leq n \leq N,$ we can use summation by parts to express
  \[
    \sum_{i} \alpha_i \epsilon_i x_i
    = \sum_{i} \alpha_i (S_{i} - S_{i-1})
    = \sum_{i} (\alpha_i - \alpha_{i+1})S_i.
  \]
  Hence as $F$ is a seminorm
  \[
    F( 
    \sum_{i} \alpha_i \epsilon_i x_i
    )
    \leq \alpha_1 \max_{1 \leq i \leq N} F(S_i)
    \leq \max_{1 \leq i \leq N} F(S_i).
  \]
  Using the reflection principle, it now follows that for any $t \geq 0$
  \[
    \Pr\left[ 
     \max_{1 \leq i \leq N} F(S_i) \geq t
    \right]
    \leq 2 
    \Pr\left[ 
     F(S_N) \geq t
    \right],
  \]
  which completes the proof (see \cite[Theorem 4.4]{LedouxTalagrand} for details).
\end{proof}

We also need the following standard Gaussian concentration inequality.
\begin{proposition}\label{prop:glip}
  Suppose that $X=\left( X_j \right)_1^n$ are i.i.d.\ standard complex Gaussian variables, and suppose $F : \C^n \to \R$ is a $1$--Lipschitz function with respect to the Euclidean metric.  Then $\Exp |X| < \infty$ and for all $t \geq 0,$ 
  \[
    \Pr\left[ X - \Exp X > t \right] \leq e^{-t^2}.
  \]
\end{proposition}
\begin{proof}
  This follows from the real case (see \cite[(1.5)]{LedouxTalagrand}).  The real and imaginary Gaussians have variance $1/2,$ for which reason the exponent is $e^{-t^2}.$
\end{proof}

\subsection*{Approximation of seminorms}

Say that a seminorm $\|\cdot \|$ on $\text{H}^2$ is \emph{approximable} if there exists a sequence of polynomials $\left\{ p_n \right\}$ with $\sup_{n,j} \| z^j \star p_n(z) \| \leq 1$
such that for all $F \in \text{H}^2,$ 
\begin{equation}\label{eq:mallorca}
  \sup_n \| F \star p_n \| < \infty  \iff  \|F \| < \infty
  \quad\text{and}\quad
  \sup_n \| F \star p_n \| = 0 \iff \|F \| =0.
\end{equation}
Let $V$ be the quotient space of $\left\{ F \in \text{H}^2 : \|F\| < \infty \right\}$ by the space $\left\{ F \in \text{H}^2 : \|F\| = 0 \right\}.$  Then both $\|\cdot\|$ and $\sup_n \| \cdot \star p_n \|$ make $V$ into Banach spaces with equivalent topologies, by the hypotheses.  Hence \eqref{eq:mallorca} is equivalent to:
\begin{equation}
  \exists~
  \text{$C > 0$ such that} 
  \quad
  \frac{1}{C}\sup_n \| F \star p_n \| \leq \|F\| \leq {C}\sup_n \| F \star p_n \|
  \quad
  \forall~%
  F \in \text{H}^2,
  \label{eq:minorca}
\end{equation}
as the inclusion map from one of these Banach spaces to the other is continuous and hence bounded.

We say that $G$ is an $H^2$--GAF if $\{a_k\} \in \ell^2$.

\begin{proposition}
  \label{prop:approximable}
  Let $F$ be an $\text{H}^2$--GAF.  Let $\|\cdot\|$ be any approximable seminorm on $\text{H}^2.$
  The following are equivalent:
  \begin{enumerate}[\textnormal{(}i\textnormal{)}]
    \item\label{it:Finas} $\|F\| < \infty$ a.s.
    \item\label{it:F1fin} $\Exp \|F\| < \infty.$
    \item\label{it:F2fin} $\Exp \|F\|^2 < \infty.$
    \end{enumerate}
\end{proposition}
\begin{remark}
  We remark that these equivalences hold in great generality for a Gaussian measure in a \emph{separable} Banach space, due to a theorem of Fernique \cite[Theorem 4.1]{Ledoux}.  As the spaces $\BMOA$ and $\Bloch$ are not separable, we instead will appeal to this notion of approximable.
\end{remark}
\begin{remark}
A priori it is not clear that a seminorm being finite is a measurable event with respect to the product $\sigma$--algebra generated by the Taylor coefficients of $G$. However, for an approximable seminorm, measurability is implied by the equivalence in \eqref{eq:mallorca}, since $ \sup_n \| F \star p_n \|$ is clearly measurable (c.f. \cite[Chapter 5, Proposition 1]{Kahane}).
\end{remark}
\begin{proof}
  The implications \eqref{it:F2fin} $\implies$ \eqref{it:F1fin} $\implies$ \eqref{it:Finas} are trivial, and so it only remains to show that \eqref{it:Finas} $\implies$ \eqref{it:F2fin}.

  Let $\left\{ p_m \right\}$ be the polynomials making $\|\cdot\|$ approximable. 
  Define
  \[
    F(z) = \sum_{k=0}^\infty a_k \xi_k z^k
    \quad
    \text{and}
    \quad
    F_m = F \star p_m.
  \]
  Without loss of generality, we may assume that $\| a\|_{\ell_2}^2 = \sum_{k=0}^\infty a_k^2 = 1.$
  For any $m \in \N,$ let $k_m = \deg(p_m),$ and define the function on $\C^{k_m}$
  \[
    G_m(x) = G_m(x_0,x_2, \dots, x_{k_m}) = \| \textstyle (\sum_{j=0}^{k_m} a_j x_j z^j) \star p_m(z) \|.
  \]
  Then for any complex vectors $x=\left( x_j \right)_0^{k_m}$ and $y=\left( y_j \right)_0^{k_m},$ by changing coordinates one at a time and by using that $\sup_{n,j} \| z^j \star p_n(z) \| \leq 1,$
  \[
    |G_m(x) - G_m(y)| 
    \leq \sum_{j=0}^{k_m} a_j |x_j - y_j|
    \leq \| a \|_{\ell_2} \| x -y \|_{\ell_2}.
  \]
  Therefore by Proposition \ref{prop:glip}, we have that for all $t \geq 0$ and all $m \in \N,$
  \begin{equation}\label{eq:seminormconcentration}
    \Pr\left[ 
    | \|F_m\| - \Exp \|F_m\| | \geq t
  \right]\leq 2e^{-t^2}.
\end{equation}
  Hence there is an absolute constant $C>0$ so that
  \begin{equation}\label{eq:clubmed}
    |\text{med}( \|F_m\|) - \Exp \|F_m\|| \leq C,
  \end{equation}
  where $\text{med}(X)$ denotes the median of a random variable $X.$

  Suppose that $\|F\| < \infty$ a.s.  By \eqref{eq:mallorca}, $\sup_m \|F_m\| < \infty$ a.s.
  Therefore there is a constant $M > 0$ such that $\Pr\left( \sup_m \|F_m\| > M \right) < \tfrac 12.$  Then
  by \eqref{eq:clubmed}
  \[
    \sup_m \Exp \|F_m\| \leq 
    \sup_m \text{med}( \|F_m\|) + C
    \leq 
    \text{med}(\sup_m  \|F_m\|) + C
    \leq M + C.
  \]
  Using \eqref{eq:minorca}, there is another absolute constant $C'$ so that
  \[
    \Exp \| F\| \leq
    C'\sup_m \Exp \|F_m\| \leq C'(M+C) < \infty.
  \]
  Using \eqref{eq:minorca} and \eqref{eq:seminormconcentration}, there is an absolute constant $C>0$
  \(
    \Var( \|F\|) \leq C, 
  \)
  and therefore 
  \[
    \Exp \| F\|^2 =  \Var( \|F\|) + (\Exp \|F\|)^2 < \infty.
  \]
\end{proof}

By Theorems 1 and 4 of \cite{Holland}, both $\starnorm$ and $\Blochnorm$ are approximable
  with $\left\{ p_n \right\}$ given by the analytic part of the Fej\'er kernel\footnote{In fact, the remark following Theorem 4 of \cite{Holland} shows that $\sup_n \starnorm[{K^A_n * f}] = \starnorm[f],$ which among other things can be used to establish the Borel measurability of $\starnorm.$}
\[
 K^A_n(z) = \sum_{k=0}^n (1-\tfrac{k}{n+1})z^k.
\]
\begin{corollary}\label{cor:expisenough}
  Let $F$ be an $\text{H}^2$--GAF. Then $\starnorm[F] < \infty$ a.s.\ if and only if $\Exp\starnorm[F] < \infty$ and 
  $\Blochnorm[F] < \infty$ a.s.\ if and only if $\Exp\Blochnorm[F] < \infty.$
\end{corollary}

We also have that the probability that a GAF is in $\BMOA,$ $\VMOA,$ or $\Bloch$ is either $0$ or $1.$

\begin{proposition}\label{prop:01}
	For any $\text{H}^2$--GAF $G$, 
	the events $\{ G \in \BMOA\}, \{ G \in \VMOA\}, \{ G \in \Bloch\}$ all have probability $0$ or $1$.
\end{proposition}
\begin{proof}
	Decompose $G = G_{\leq n} + G_{>n},$ where $G_{\leq n}$ is the $n$--th Taylor polynomial of $G$ at $0.$  Then as $G_{\leq n}$ is a polynomial, $\starnorm[ G_{\leq n}] < \infty$ almost surely.
	Hence $\starnorm[ G ] < \infty$ if and only if $ \starnorm[ G_{> n} ] < \infty,$ up to null events.  Therefore, $\starnorm[ G ] < \infty$ differs from a tail event of $\left\{ \xi_n : 1\leq n < \infty \right\}$ by a null event, and so the statement follows from the Kolmogorov $0$-$1$ law.  The same proof shows that $\Pr\left[ G \in \Bloch \right] \in \left\{ 0,1 \right\}$.

	For VMOA, as $G_{\leq n}$ is a polynomial,
	\[
	\lim_{|I| \to 0} \sup_{I} M_I^1(G_{ \leq n}) = 0 \quad \As,
	\]
	and the same reasoning as above gives the $0$-$1$ law.
\end{proof}

\section{The Sledd space}
\label{sec:sledd}

Let $K_n$ for $n \in \N$ be the $n$--th Fej\'er kernel, which for $|z|=1$ is given by
\begin{equation}\label{eq:fejer}
  K_n(z) = \sum_{k=-n}^n (1-\tfrac{|k|}{n+1})z^k = \frac{1}{n+1}\cdot\frac{|1-z^{n+1}|^2}{|1-z|^2}.
\end{equation}
This kernel has the two familiar properties $\| K_n\|_1 = 1$ and $K_n(z) \leq \frac{4}{n+1}\cdot\frac{1}{|1-z|^2}.$

For a function $F : \T \to \C$ with a Laurent expansion on $\T,$ let $\widehat F : \Z \to \C$ be its Fourier coefficients, i.e.\ let $\widehat{F}(k)$ be the $k$--th coefficient of its Laurent expansion.

We let $T_n$ be the dyadic trapezoidal kernel
\begin{equation}
  \begin{aligned}
  &T_0(z) = 1 + \tfrac{1}{2}z+\tfrac{1}{2}z^{-1} \\
  &T_n = 2K_{2^{n+2}} - K_{2^{n+1}} + K_{2^{n-1}} - 2K_{2^n}, \quad n \geq 1.
\end{aligned}
  \label{eq:trapezoid}
\end{equation}
The kernel $T_n$ satisfies that $\widehat T_n$ is supported in $[2^{n-1}, 2^{n+2}),$ has $|\widehat T_n(K)| \leq 1$ everywhere, has $\widehat T_n(K) = 1$ for $K \in [2^{n}, 2^{n+1}]$ and satisfies
\[
  \sum_{n=0}^\infty \widehat T_n(K) = 1
\]
for all integers $K \geq 0.$  Further, $\|T_n\|_1 \leq 6$ for all $n \geq 0.$
Also
\begin{equation}
  |T_n(z)| \leq 20 \cdot 2^{-n} |1-z|^{-2}. 
  \label{eq:trapezoidbound}
\end{equation}
Recall that in terms of the kernels $\left\{ T_n \right\},$ we defined the seminorm (in \eqref{eq:TSleddintro}),
\begin{equation}\label{eq:TSledd}
  \TSledd[F]^2 = \sup_{|x|=1} \sum_{n=0}^\infty | T_n \star F(x)|^2.
\end{equation}
In \cite{sledd1981random}, it is shown that this norm is related to $\starnorm$ in the following way:
\begin{theorem} \label{thm:sledd}
  If $F \in \text{H}^1,$ then there is an absolute constant $C>1$ so that
  \[
    \|F\|_{*} \leq C\TSledd[F]. %
  \]
\end{theorem}
Sledd also gives a sufficient condition for $F$ to be in $\VMOA$, though we observe that there is a stronger one that follows directly from Theorem \ref{thm:sledd}.
\begin{theorem}
  If $F \in \text{H}^1$ and if
  \[
    \lim_{k \to \infty} \sup_{|x|=1} \sum_{n=k}^\infty | T_n \star F(x)|^2  = 0
  \]
  then $F \in \VMOA.$
  \label{thm:tailsledd}
\end{theorem}
\begin{proof}
  The space $\VMOA$ is the closure of continuous functions in the $\BMOA$ norm.  Hence it suffices to find, for any $\epsilon > 0,$ a decomposition $G = G_1 + G_2$ with $G_1$ continuous and $\|G_2\|_{\BMOA} \leq \epsilon.$   For any $\epsilon > 0,$ we may by hypothesis pick $k$ sufficiently large that
  \[
     \sup_{|x|=1} \sum_{n=k}^\infty | T_n \star G(x)|^2
     \leq \epsilon.
  \]
  Using Theorem \ref{thm:sledd}, it follows that if we decompose 
  \[
    G = G_1 + G_2, \quad \text{where} \quad
    G_1 = \sum_{n=0}^{k-1} T_n \star G
    \quad
    \text{and}
    \quad
    G_2 = \sum_{n=k}^{\infty} T_n \star G.
  \]
  Then $G_1$ is a polynomial and in particular continuous.  From the properties of the Fourier support of $\left\{ T_n \right\},$ 
  \begin{equation}
    T_n \star G_2
    =
    \begin{cases}
      T_n \star G,& \text{ if  } n \geq k+2, \\
      \sum_{p=k}^{k+3} T_n \star T_{p} \star G,& \text{ if  } k-2 \leq n \leq k+1, \text{ and } \\
      0,& \text{ if  } n \leq k-3. \\
    \end{cases}
    \label{eq:TTmaster}
  \end{equation}
  Thus, we have for any $n \in [k-2,k+1]$ by using $\|T_n\|_1 \leq 6$ and convexity of the square, that 
  \[
    \|T_n \star G_2\|_{\infty}^2
    \lesssim
    \sup_{n \geq k}
    \|T_n \star G\|_{\infty}^2
    \leq
    \sup_{|x|=1} \sum_{n=k}^\infty | T_n \star G(x)|^2 \leq \epsilon,
  \]
  Applying Theorem \ref{thm:sledd} to $G_2$ and using the properties derived in \eqref{eq:TTmaster}
  \[
    \begin{aligned}
    \starnorm[G_2]^2
    \lesssim
    \sup_{|x|=1} \sum_{n=0}^\infty | T_n \star G_2(x)|^2
    &=
    \sup_{|x|=1} \sum_{n=k}^\infty | T_n \star G_2(x)|^2 \\
    &\leq
    \sup_{|x|=1} \sum_{n=k+2}^\infty | T_n \star G(x)|^2
    +
    \sum_{n=k-2}^{k+1}\|T_n \star G_2\|_\infty^2
    \lesssim 
    \epsilon.
    \end{aligned}
  \]

\end{proof}

\begin{proposition}\label{prop:sleddnonsep}
The Sledd space $\Sledd$ is non-separable.
\end{proposition}
\begin{proof}[Sketch of the proof]
We sketch the construction of an uncountable family of analytic functions in $\Sledd$ whose pairwise distances in $\TSledd$ are uniformly bounded below. Put
\[
G_j(z) = \frac1{2^j + 1} z^{2^{j+1}} K_{2^j} (z e(1/j) ),\qquad j \ge 1.
\]
Notice that $\widehat{G}_j$ is supported in $[2^{j}, 2^{j+2}]$, and that $G_j$ has the following properties:
\begin{enumerate}
	\item $|G_j(e(-1/j))| = 1$,
	\item $|G_j(e(\theta))| \le 1$ for all $\theta$,
	\item $|G_j(e(-1/j + \theta))| \lesssim 2^{-j}$ when $c 2^{-j/2} \le |\theta| \le \pi$.
\end{enumerate}

For any $A \subset 5 \bbn$ let $H_A = \sum_{n \in A} G_n$. By the above properties all these functions belong to $\Sledd$, and are uniformly separated from each other.
\end{proof}

\begin{remark}
The construction above gives an example of functions in $\Sledd$ which are not continuous on the boundary of the disk.
\end{remark}

\subsection*{GAFs and the Sledd space}

We shall be interested in applying Sledd's condition to GAFs, for which purpose it is possible to make some simplifications.  For any $n \geq 0,$ let $R_n$ be the kernel defined by
\[
  \widehat{R_n}(K) = 
  \begin{cases}
    1 & \text{if } K \in [2^n, 2^{n+1}), \\
  0 & \text{otherwise}.
  \end{cases}
\]
In short, for a GAF, (and more generally any random series with symmetric independent coefficients) we may replace the trapezoidal kernel $T_n$ by $R_n;$ specifically:
\begin{theorem} \label{thm:tailsleddequiv}
  Suppose $G$ is an $\text{H}^2$-GAF.  Then the following are equivalent:
  \begin{enumerate}[\textnormal{(}i\textnormal{)}]
    \item 
    \(
      \lim_{k \to \infty} \sup_{|x|=1} \sum_{n=k}^\infty | T_n \star G(x)|^2  = 0 \quad \As
    \)
    \item
      \(
	\lim_{k \to \infty} \Exp \biggl[ \sup_{|x|=1} \sum_{n=k}^\infty | T_n \star G(x)|^2 \biggr]  = 0. 
    \)
    \item
      \(
	\lim_{k \to \infty} \Exp \biggl[ \sup_{|x|=1} \sum_{n=k}^\infty | R_n \star G(x)|^2 \biggr]  = 0.
    \)
    \item 
    \(
      \lim_{k \to \infty} \sup_{|x|=1} \sum_{n=k}^\infty | R_n \star G(x)|^2  = 0 \quad \As
    \)
  \end{enumerate}
\end{theorem}

\begin{proof}[Proof of Theorem \ref{thm:tailsleddequiv}]
  We begin with the equivalence of $(ii)$ and $(iii),$ and the implication that $(iii)$ implies $(ii).$
  For any $n \geq 0$ and any $j \in \{1,2,3,4\}$ define $R_{n,j} = T_n \star R_{n+j-1}.$  Then $T_n = \sum_{j=1}^4 R_{n,j}.$ 
  Using convexity, we can bound
  \[
    \sup_{|x|=1} \sum_{n=k}^\infty | T_n \star G(x)|^2
    \lesssim
    \sum_{j=1}^4
    \sup_{|x|=1} \sum_{n=k}^\infty | R_{n,j} \star G(x)|^2.
  \]
  Since $\widehat R_{n,j}$ is supported in  $[2^n,2^{n+1})$ and has $\|\widehat R_{n,j}\|_\infty \leq 1,$ the contraction principle implies that for any $0 \leq k \leq m < \infty,$
    \[
      \Exp \sup_{|x|=1} \sum_{n=k}^m | R_{n,j} \star G(x)|^2
      \leq \Exp \sup_{|x|=1} \sum_{n=k}^m | R_{n} \star G(x)|^2
      \leq \Exp \sup_{|x|=1} \sum_{n=k}^\infty | R_{n} \star G(x)|^2.
    \]
  Sending $m \to \infty$ and using monotone convergence implies that
  \[
      \Exp \sup_{|x|=1} \sum_{n=k}^\infty | R_{n,j} \star G(x)|^2
      \leq \Exp \sup_{|x|=1} \sum_{n=k}^\infty | R_{n} \star G(x)|^2,
  \]
  from which the desired convergence follows.

  Conversely, to see that $(ii)$ implies $(iii),$ we begin by bounding
  \[
    \sup_{|x|=1} \sum_{n=k}^\infty | R_{n} \star G(x)|^2
    \leq
    \sum_{j=1}^4 
    \sup_{|x|=1} \sum_{\substack{n \geq k \\ n \in 4\N + j} }^\infty | R_{n} \star G(x)|^2. 
  \]
  Then by the contraction principle and monotone convergence for any $j \in \left\{ 1,2,3,4 \right\}$
  \[
    \Exp
     \sup_{|x|=1} \sum_{\substack{n \geq k \\ n \in 4\N + j} }^\infty | R_{n} \star G(x)|^2
    \leq 
    \Exp
     \sup_{|x|=1} \sum_{\substack{n \geq k \\ n \in 4\N + j} }^\infty | T_{n} \star G(x)|^2
    \leq 
    \Exp
    \sup_{|x|=1} \sum_{\substack{n \geq k} }^\infty | T_{n} \star G(x)|^2,
  \]
  which completes the proof of the desired implication.

  We turn to showing the equivalence of $(i)$ and $(ii).$
  From Markov's inequality, $(ii)$ implies that
  \[
     \sup_{|x|=1} \sum_{n=k}^\infty | T_n \star G(x)|^2
     \Prto[k] 0.
  \]
  As the sequence
  \(
     \sup_{|x|=1} \sum_{n=k}^\infty | T_n \star G(x)|^2
  \)
  is monotone and therefore always converges, it follows it converges almost surely to $0.$

  Define for each $k \in \N$ the seminorms
  \[
    \begin{aligned}
      &\RSledd<k> : \text{H}^1 \to [0,\infty], \quad\text{where}\quad \RSledd[f]<k>^2 \coloneqq \sup_{|x|=1} \sum_{n=k}^\infty | R_n \star f(x)|^2 \\
      &\TSledd<k> : \text{H}^1 \to [0,\infty], \quad\text{where}\quad \TSledd[f]<k>^2 \coloneqq \sup_{|x|=1} \sum_{n=k}^\infty | T_n \star f(x)|^2
  \end{aligned}
  \]
  In preparation to use Proposition \ref{prop:approximable}, we show the following claim.
  \begin{claim}\label{claim:approximable}
    The seminorms $\left\{ \RSledd<k>, \TSledd<k> \right\}$ are approximable.
  \end{claim}
  \begin{proof}
  Let $p_m$ be the polynomial whose coefficients are $1$ for coefficients $0$ to $2^{m+1}-1.$  Then for any $m > k,$
  \[
    \RSledd[p_m \star f]<k>^2 = \sup_{|x|=1} \sum_{n=k}^m | R_n \star f(x)|^2
    \xrightarrow[m \to \infty]{} \RSledd[f]<k>^2.
  \]
  Let $q_m(z)$ be the sum of analytic part of $\sum_{k=0}^m T_k(z).$ Then for analytic $f$ in the disk, $q_m \star f = \sum_{k=0}^m T_k \star f.$  Moreover, using \eqref{eq:trapezoid} the sum $\sum_{k=0}^m T_k$ can be represented by a sum of a finite number of Fej\'er kernels with cardinality bounded independent of $k.$  Therefore there is an absolute constant $C >0$ so that for all $m$
  \begin{equation}\label{eq:fejerispositive}
    \|q_m \star f\|_\infty \leq \|{\textstyle \sum_{k=0}^m T_k}\|_1 \|f\|_\infty \leq C \|f\|_\infty.
  \end{equation}
  Using that $\widehat{q_m}(j) = 1$ for $0 \leq j \leq 2^m-1$
  \[
    \TSledd[q_m \star f]<k>^2
    \geq 
    \sup_{|x|=1} \sum_{n=k}^{m-2} | T_n \star f(x)|^2
    \xrightarrow[m \to \infty]{} \TSledd[f]<k>^2,
  \]
  and so if $\sup_m \TSledd[q_m \star f]<k>^2 < \infty$ then so is $\TSledd[f]<k>^2.$
  Conversely, if $\TSledd[f]<k>^2<\infty$ then $\sup_{n \geq k}  \| T_n \star f \|_\infty < \infty$, and hence with the same $C$ as in \eqref{eq:fejerispositive},
  \[
    \max_{m-1 \leq n \leq m+2} \|q_m \star T_n \star f\|_\infty \leq C\TSledd[f]<k>.
  \]
  So
  \[
    \TSledd[q_m \star f]<k>^2
    \leq
    \sup_{|x|=1} \sum_{n=k}^{m-2} | T_n \star f(x)|^2
    +
    \sum_{n=m-1}^{m+2}
    \|q_m \star T_n \star f\|_\infty^2
    \leq (1+4C^2)\TSledd[f]<k>^2 < \infty.
  \]
  \end{proof}

  We show the equivalence of (iii) and (iv).  The proof of the equivalence of (i) and (ii) is the same.
  From (iii) it follows from Markov's inequality that
  \[
    \sup_{|x|=1} \sum_{n=k}^\infty | R_n \star G(x)|^2  \Prto[k] 0
  \]
  By monotonicity $\sup_{|x|=1} \sum_{n=k}^\infty | R_n \star G(x)|^2$ converges almost surely, and so it converges almost surely to $0.$
  From (iv) and by Claim \ref{claim:approximable}, there exists a $k_0$ so that
  \[
    \Exp \sup_{|x|=1} \sum_{n=k_0}^\infty | R_n \star G(x)|^2 < \infty.
  \]
  As a consequence, it is possible to take $k_0=0.$
  By dominated convergence
  \[
    \lim_{k\to \infty} \Exp \biggl[ \sup_{|x|=1} \sum_{n=k}^\infty | R_n \star G(x)|^2 \biggr] = 0.
  \]

 \end{proof}

\begin{remark}
  \label{rem:sleddequiv}
In reviewing the proof, one also sees that under the same assumptions, the following are equivalent:
\begin{enumerate}[\textnormal{(}i\textnormal{)}]
    \item 
    \(
      \sup_{|x|=1} \sum_{n=0}^\infty | T_n \star G(x)|^2  < \infty \quad \As
    \)
    \item
      \(
	\Exp \biggl[ \sup_{|x|=1} \sum_{n=0}^\infty | T_n \star G(x)|^2 \biggr]  < \infty. 
    \)
    \item
      \(
	\Exp \biggl[ \sup_{|x|=1} \sum_{n=0}^\infty | R_n \star G(x)|^2 \biggr]  <\infty.
    \)
    \item 
    \(
      \sup_{|x|=1} \sum_{n=0}^\infty | R_n \star G(x)|^2  < \infty \quad \As
    \)
  \end{enumerate}
  Moreover, the proof gives that there is an absolute constant $C>0$ so that
  \[
    \frac{1}{C}
    \Exp
    \RSledd[G]^2
    \leq
    \Exp
    \TSledd[G]^2
    \leq
    C
    \Exp
    \RSledd[G]^2.
  \]
\end{remark}

Finally, we show that for a GAF, finiteness of $\RSledd[G]$ in fact implies $G \in \VMOA.$ 

\begin{theorem}
  If $G$ is an $\text{H}^2$-GAF for which
  \[ 
    \RSledd[G]^2 = \sup_{|x|=1} \sum_{n=0}^\infty | R_n \star G(x)|^2 <\infty \quad \As,
  \]
  then
  \[
      \lim_{k \to \infty} \sup_{|x|=1} \sum_{n=k}^\infty | R_n \star G(x)|^2  = 0 \quad \As
  \]
  Furthermore, $\RSledd[G] {} < \infty$ implies $G$ is in $\VMOA.$
  \label{thm:sleddvmoa}
\end{theorem}

We will need the following result \cite[Chapter 5, Proposition 12]{Kahane}.
\begin{proposition}
  \label{prop:rotations}
  Let $u_1,u_2,\dots,$ be a sequence of continuous functions on the unit circle such that $\limsup_{k \to \infty} \|u_k\|_\infty > 0.$ Let $\theta_1, \theta_2, \dots$ be a sequence of independent random variables uniformly distributed on $[0,1].$  Then almost surely there exists a $t \in [0,1]$ such that
  \(
\limsup_{k \to \infty} |u_k(e(t-\theta_k))| > 0.
\)
\end{proposition}

\begin{proof}[Proof of Theorem \ref{thm:sleddvmoa}]
Let $v_n \coloneqq |R_n \star G|^2$ for all $n \geq 1.$  Suppose to the contrary
  \[
      V \coloneqq \lim_{k \to \infty} \sup_{|x|=1} \sum_{n=k}^\infty v_n(x)
  \]
  is not almost surely $0.$  Then as $V$ is tail--measurable, there is a $\delta \in (0,1)$ so that $V > \delta$ \As\ \ By monotonicity, it follows that for all $k$
  \[
    \sup_{|x|=1} \sum_{n=k}^\infty v_n(x) > \delta \quad \As
  \]
  Furthermore, deterministically
  \[
    \lim_{m \to \infty} \sup_{|x|=1} \sum_{n=k}^m v_n(x)
    =\sup_{|x|=1} \sum_{n=k}^\infty v_n(x).
  \]
  By continuity of measure,
  \[
    \lim_{m \to \infty} \Pr\biggl( \sup_{|x|=1} \sum_{n=k}^m v_n(x) > \delta\biggr)
    =\Pr\biggl(\lim_{m \to \infty} \sup_{|x|=1} \sum_{n=k}^m v_n(x) > \delta\biggr) = 1.
  \]
  Thus there is a sequence $m_1 < m_1' < m_2 < m_2^\prime < \dots$ so that if 
  \(
  u_k \coloneqq \sum_{n=m_k}^{m_k'} v_n,
  \)
  then
  \[
    \Pr( \|u_k\|_\infty > \delta) > \delta.
  \]
  By Borel--Cantelli
  \[
    \Pr( \limsup_{k \to \infty} \|u_k\|_\infty > \delta) = 1.
  \]
  Let $\left\{ \theta_k \right\}$ be i.i.d.\ uniform variables on $[0,1]$ which are also independent of $G.$  
  Therefore by conditioning on $G$ and using Proposition \ref{prop:rotations}
  there is almost surely a $t \in [0,1]$ so that
  \[
    \limsup_{k \to \infty} v_k(e(t-\theta_k)) > 0.
  \]
  Because $\left\{ v_n(xe(\theta_k)) \right\}$ has the same distribution as $\left\{ v_n(x) \right\},$
  it follows there is almost surely a $s \in [0,1]$ so that
  \[
    \limsup_{k \to \infty} v_k(e(s)) > 0.
  \]
  Therefore $\RSledd[G]^2 \geq V = \infty$ \As, which concludes the first part of the proof.

  Using Theorem \ref{thm:tailsledd}, Theorem \ref{thm:tailsleddequiv} and Remark \ref{rem:sleddequiv}, the second conclusion follows.
\end{proof}

\section{Sufficient condition for a GAF to be Sledd}.
\label{sec:sufficiency}

In this section we will give a sufficient condition on the coefficients of the GAF to be in $\Sledd.$ We begin with the following preliminary calculation.
\begin{lemma}
  Let $(H_1,H_2)$ be a centered complex Gaussian vector with $\Exp |H_1|^2 =\Exp |H_2|^2 = 1$ and $\bigl|\Exp [H_1 \overline{H_2}]\bigr| = \rho \in [0,1].$  Then for all $|\lambda| < {(1-\rho^2)^{-1/2}},$
  \[
    \Exp e^{ \lambda( |H_1|^2 - |H_2|^2 )} = \frac{1}{1-\lambda^2(1-\rho^2)}.
  \]
  \label{lem:magic}
\end{lemma}
\begin{proof}
  We may assume without loss of generality that $\Exp [H_1 \overline{H_2}] = \rho \geq 0.$
  Hence, we may write
  \[
    \begin{pmatrix}
      H_1 \\ 
      H_2
    \end{pmatrix}
    =
    A
    \begin{pmatrix}
      Z_1 \\
      Z_2
    \end{pmatrix}
    \coloneqq
    \begin{pmatrix}
      1 & 0 \\
      \rho & \sqrt{1-\rho^2} \\
    \end{pmatrix}
    \begin{pmatrix}
      Z_1 \\
      Z_2
    \end{pmatrix},
  \]
  where $Z = (Z_1, Z_2)$ are independent standard complex normals, considered as a column vector.  Therefore,
  \[
     |H_1|^2 - |H_2|^2
     =
     Z^* A^*
    \begin{pmatrix}
      1 & 0 \\
      0 & -1 \\
    \end{pmatrix}
     A Z.
  \]
  It follows that
  \[
    \Exp e^{ \lambda( |H_1|^2 - |H_2|^2 )} = \frac{1}{\det(\Id - \lambda A^*
    \left(\begin{smallmatrix}
      1 & 0 \\
      0 & -1 \\
    \end{smallmatrix}\right)
     A) }
     =
     \frac{1}{1-\lambda^2(1-\rho^2)}.
  \]
\end{proof}
We shall apply this equality to the complex Gaussian process $Q_n(\theta) \coloneqq R_n \star G(e(\theta)).$  Then
  \[
    \sigma_n^2 = \Exp | Q_n |^2 
    \quad
    \text{ and define }
    \quad
    \rho_n \coloneqq \rho_n(\theta_1-\theta_2) \coloneqq \sigma_n^{-2} |\Exp[Q_n(\theta_1)\,\overline{ Q_n(\theta_2) }]| \in [0,1].
  \]
  In the case that $\sigma_n^2 = 0,$ we may take any value in $[0,1]$ for $\rho_n.$
  From Lemma \ref{lem:magic}, for any $|\lambda|^2 < (1-\rho_n^2)^{-1}\sigma_n^{-4},$ we have that
  \begin{equation}\label{eq:mgf}
    \Exp \exp\left( \lambda(| Q_n(\theta_1)|^2-| Q_n(\theta_2)|^2) \right)
    =
    \frac{1}{1-\lambda^2(1-\rho_n^2)\sigma_n^4}.
  \end{equation}
  While we would like to use $\sigma_n^4(1-\rho_n^2(\theta_1-\theta_2))$ as a distance, it does not obviously satisfy the triangle inequality, for which reason we introduce:  
\begin{equation}\label{eq:Delta}
  \Delta_n(\theta) \coloneqq \Exp \bigl| |Q_n(\theta)|^2 - |Q_n(0)|^2\bigr|,
\end{equation}
which defines a pseudometric on $[0,1]$ through $\Delta_n(\theta_1,\theta_2) \coloneqq \Delta_n(\theta_1 - \theta_2).$  While $\Delta_n$ may not obviously control the tails of $|Q_n(\theta)|^2$, we observe the following lemma.
\begin{lemma}\label{lem:allgood}
  There is a numerical constant $C > 1$ so that for all choices of $\left\{ a_k \right\}$ and any $n \geq 0$ and all $\theta \in [0,1]$
\[
  \frac{1}{C}\sigma_n^2\sqrt{1-\rho_n^2(\theta)} \leq \Delta_n(\theta) \leq C \sigma_n^2\sqrt{1-\rho_n^2(\theta)}.
\]
\end{lemma}
\begin{proof}
  From \eqref{eq:mgf}, it follows that
  \[
    \begin{aligned}
    &\Exp (|Q_n(\theta)|^2 - |Q_n(0)|^2)^2
    = \sigma_n^4(1-\rho_n^2) \\
    &\Exp (|Q_n(\theta)|^2 - |Q_n(0)|^2)^4
    = 2\sigma_n^8(1-\rho_n^2)^2.
    \end{aligned}
  \]
  Hence by Cauchy--Schwarz,
  \[
    \Delta_n^2(\theta)
    \leq \sigma_n^4(1-\rho_n^2).  
  \]
  On the other hand, by the Paley--Zygmund inequality,
  \[
    (|Q_n(\theta)|^2 - |Q_n(0)|^2)^2
    \geq \frac{1}{100} \sigma_n^4(1-\rho_n^2)
  \]
  with probability at least $1/2$ which gives a lower bound for $\Delta_n$ of the same order.
\end{proof}

We now define two pseudometrics on $[0,1]$ in terms of $\left\{ \Delta_n \right\}.$
\begin{equation}\label{eq:metrics}
  \begin{aligned}
  &d_\infty(\theta_1, \theta_2)
  \coloneqq
  d_\infty(\theta_1-\theta_2)
  \coloneqq
  \sup_{n  \geq 0} \Delta_n(\theta_1-\theta_2), \text{ and } \\
  &d_2^2(\theta_1, \theta_2)
  \coloneqq
  d_2^2(\theta_1-\theta_2)
  \coloneqq
  \sum_{n  \geq 0} \Delta_n^2(\theta_1-\theta_2).
\end{aligned}
\end{equation}
Using Lemma \ref{lem:magic}, we can also now give a tail bound for differences of 
\begin{equation}\label{eq:F}
  F(\theta)
  \coloneqq
  \sum_{n=0}^\infty |Q_n(\theta)|^2.
\end{equation}
\begin{lemma}
  Let $\theta_1,\theta_2 \in [0,1].$  
 There is a numerical constant $C > 1$ so that
  for all $t \geq 0,$
  \[
  \begin{aligned}
    &
    \Pr\left[ F(\theta_1) - F(\theta_2) \geq t \right] \leq \exp\left( 
    -C\min \left\{ 
      \frac{t}{d_\infty(\theta_1-\theta_2)}
      ,
      \frac{t^2}{d_2^2(\theta_1-\theta_2)}
    \right\}
    \right) 
  \end{aligned}
\]
\label{lem:kobe}
\end{lemma}
\begin{proof}

  The desired tail bound follows from estimating the Laplace transform of $F(\theta_1) - F(\theta_2).$  Specifically we use the following estimate: 
\begin{lemma} \label{lem:expintegral}
Suppose that there are $\lambda_0,$ $\sigma >0$ and $X$ a real valued random variable for which  
\[
  \Exp e^{\lambda X} \leq e^{\lambda^2 \sigma^2/2} \quad \text{ for } \lambda^2 \leq \lambda_0^2,
\]
then for all $t \geq 0$
\[
  \Pr\left[  X \geq t \right] \leq \exp\left( 
  -\min\left\{ \frac{\lambda_0 t}{2}, \frac{t^2}{2\sigma^2} \right\}
  \right).
\]
\end{lemma}
\begin{proof}
  Applying Markov's inequality, for any $t \geq 0$ and $0 < \lambda \leq \lambda_0,$
  \[
  \Pr\left[  X \geq t \right] \leq \exp\left( 
  -\lambda t + \lambda^2 \sigma^2/2
  \right).
  \]
  Taking $\lambda = t/\sigma^2,$ if possible, gives one of the bounds.  Otherwise, for $\lambda_0 \leq t/\sigma^2,$ taking $\lambda = \lambda_0$ gives the other bound.
\end{proof}

We return to estimating the Laplace transform of  $F(\theta_1) - F(\theta_2).$
Recalling \eqref{eq:mgf}, for any 
\[
  |\lambda|^2 < \lambda_\star^2 := \inf_{n \in \N} (1-\rho_n^2)^{-1} \sigma_n^{-4} \leq \frac{C^2}{d_\infty(\theta_1-\theta_2)^2},
\]
where $C$ is the numerical constant from Lemma \ref{lem:allgood},
we have
  \begin{equation}\label{eq:baller}
    \Exp \exp\left( \lambda( F(\theta_1) - F(\theta_2)) \right)
    =
    \prod_{n=1}^\infty  
    \frac{1}{1-\lambda^2(1-\rho_n^2)\sigma_n^4}.
  \end{equation}
  Therefore, for all $|\lambda|^2 < \lambda_\star^2/2$
  \begin{equation} \label{eq:smalll}
    \Exp \exp\left( \lambda( F(\theta_1) - F(\theta_2)) \right)
    \leq
    \prod_{n=1}^\infty  
    \frac{1}{1-\lambda^2(1-\rho_n^2)\sigma_n^4}
    \leq 
    \exp\left( 2\lambda^2 \sum_{n=1}^\infty (1-\rho_n^2)\sigma_n^4 \right).
  \end{equation}
  using $(1-x)^{-1} \leq e^{2x}$ for $0 \leq x \leq \tfrac12.$  The desired conclusion now follows from Lemma \ref{lem:allgood} and Lemma \ref{lem:expintegral} .
\end{proof}

  We now recall the technique of Talagrand for controlling the supremum of processes.  We let $T = [0,1].$ Define for any metric $d$ on $T$ and any $\alpha \geq 1$
  \begin{equation}
    \gamma_{\alpha}(d) = \inf \sup_{t \in T} \sum_{k \geq 0} d(t,C_k) 2^{k/\alpha},
    \label{eq:talagrandgamma}
  \end{equation}
  where the infimum is taken over all choices of finite subsets $(C_k)_{k \geq 0}$ of $T$ with cardinality $|C_k| = 2^{2^k}$ for $k \geq 1$ and $|C_0| = 1.$
  \begin{theorem}[{See \cite[Theorem 1.3]{talagrand2001}}]
    \label{thm:talagrand}
    Let $d_\infty$ and $d_2$ be two pseudometrics on $T$ and let $(X_t)_{t \in T}$ be a process so that
    \[
      \Pr\left[ 
      |X_s - X_t| \geq u
      \right] \leq 
      2\exp\left( 
      -\min\left\{ 
	\frac{u}{d_\infty(s,t)},
	\frac{u^2}{d_2^2(s,t)}
      \right\}
      \right).
    \]
    Then there is a universal constant $C>0$ so that
    \[
      \Exp \sup_{s,t \in T} |X_s - X_t| \leq C \left(  \gamma_1(d_\infty) +  \gamma_2(d_2) \right).
    \]
  \end{theorem}
  Hence, as an immediate corollary of this theorem and of Lemma \ref{lem:kobe}, we have:
  \begin{corollary}\label{cor:finiteness}
    There is a numerical constant $C>0$ so that
    \[
      \Exp \sup_\theta F(\theta)
      \leq C(\gamma_1(d_\infty) + \gamma_2(d_2)) + \sqrt{\textstyle \sum a_n^2} .%
    \]
  \end{corollary}

  Finally, we give some estimates on the quantities $\gamma_1$ and $\gamma_2$ for the metrics we consider.  We begin with an elementary observation that shows  $\Delta_n(\theta)$ must decay for sufficiently small angles (when $|\theta| \leq 2^{-n}.$)
  \begin{lemma}\label{lem:toosmall}
    There is a numerical constant $C>1$ so that for all $\theta \in [-1,1]$
    \[
      1-\rho_n^2(\theta) \leq C2^{2n} |\theta|^2
      \quad\text{and}\quad
      \Delta_n(\theta) \leq C\sigma_n^22^{n}|\theta|.
    \]
  \end{lemma}
  \begin{proof}
  We begin by observing that $\rho_n$ can always be bounded by
  \[
    \begin{aligned}
    \rho_n
    &\geq \sigma_n^{-2}\cdot \sum_{k=2^n}^{2^{n+1}-1} |a_k|^2 \cos(2\pi k(\theta)) \\
    &\geq 1 - 2\pi^2 2^{n+2} \theta^2.
  \end{aligned}
  \]
  The proof now follows from Lemma \ref{lem:allgood}.
  \end{proof}

We now show that $\Exp \RSledd[G]^2$ has the desired control.  For any $k \geq 0,$ let
\[
\tau_k^2 = \sup_{n\geq k}\sigma_n^2.
\]
\begin{lemma}\label{lem:g2sucks}
	There is an absolute constant $C>0$ so that
	\[
	\Exp \RSledd[G]^2  \leq C{\textstyle \sum \tau_k^2}.
	\]
\end{lemma}
\noindent This lemma %
 proves Theorem \ref{thm:bettersledd}.
  \begin{proof}
  	From Corollary \ref{cor:finiteness},
  	\[
  	\Exp \RSledd[G]^2 \lesssim \gamma_1(d_\infty) + \gamma_2(d_2) + {\textstyle \sum \tau_k^2}.
  	\]
  We will choose $C_k$ to be the dyadic net $\left\{ \ell 2^{-2^k} : 1 \leq \ell \leq 2^{2^k} \right\}.$  Then using Lemma \ref{lem:toosmall} it follows that for any $t \in [0,1]$
  \begin{equation}\label{eq:ds}
    \begin{aligned}
      &d_\infty(t,C_k) = d_\infty(2^{-2^k}) \lesssim \sup_{n \geq 0} \left\{ (1-\rho_n^2(2^{-2^k}))\sigma_n^2 \right\} 
      \lesssim \sup_{n \geq 0}\left\{ 2^{-(n-2^k)_{-}}\sigma_n^2 \right\},\\
      &d_2^2(t,C_k) = d_2^2(2^{-2^k}) \lesssim \sum_{n=0}^{\infty} (1-\rho_n^2(2^{-2^k}))^2\sigma_n^4
      \lesssim \sum_{n=0}^\infty\left\{ 2^{-2(n-2^k)_{-}}\sigma_n^4 \right\}.
  \end{aligned}
  \end{equation}
  In the previous equations, $x_{-} \coloneqq -\min\{x,0\}.$ 
  
 This leads to an estimate on $\gamma_1$ of 
 \begin{equation}\label{eq:gamma1}
 \gamma_1(d_1)
 \leq \sum_{k=0}^\infty \left\{\sup_{n \geq 0}\left\{ 2^{-(n-2^k)_{-}}\sigma_n^2 \right\}\right\}\cdot 2^k
 \end{equation}
 and an estimate
 \begin{equation}\label{eq:gamma2}
 \gamma_2(d_2)
 \leq \sum_{k=0}^\infty \sqrt{\left\{\sum_{n \geq 0}\left\{ 2^{-2(n-2^k)_{-}}\sigma_n^4 \right\}\right\}} \cdot 2^{k/2}.
 \end{equation}
 
 We show that
 \begin{equation}\label{eq:g1g2}
 \gamma_1(d_1) + \gamma_2(d_2) 
 \lesssim \sum_{k=0}^\infty \tau_k^2.
 \end{equation}
 To control $\gamma_1(d_1),$ we 
 begin by applying Cauchy condensation
 \begin{equation}\label{eq:g11}
 \gamma_1(d_1)
 \lesssim \sum_{k=0}^\infty \left\{\sup_{n \geq 0}\left\{ 2^{-(n-k)_{-}}\sigma_n^2 \right\}\right\}.
 \end{equation}
 We then estimate
 \[
 \sup_{n \geq 0}\left\{ 2^{-(n-k)_{-}}\sigma_n^2 \right\}
 \leq
 \sum_{n=0}^{k} \{ 2^{n-k} \sigma_n^2 \}
 +
 \tau_{k}^2.
 \]
 Applying this bound, and changing the order of summation for the first,
 it follows that $\gamma_1(d_1) \lesssim \sum_k \tau_k^2.$
 
 To control $\gamma_2(d_2),$ we again begin by applying Cauchy condensation which results in 
 \begin{equation}\label{eq:g22}
 \gamma_2(d_2)
 \leq \sum_{k=0}^\infty \sqrt{\biggl\{\sum_{n \geq 0}\left\{ 2^{-2(n-k)_{-}}\sigma_n^4 \right\}\biggr\}\cdot \frac{1}{k}}.
 \end{equation}
 We then split the sum
 \begin{equation}\label{eq:g23}
 \gamma_2(d_2)
 \leq 
 \sum_{k=0}^\infty \sqrt{\biggl\{\sum_{0 \leq n \leq k}\left\{ 2^{2(n-k)}\tau_n^4 \right\}\biggr\}\cdot \frac{1}{k}}
 +\sum_{k=0}^\infty \sqrt{\biggl\{\sum_{n \geq k}\tau_n^4 \biggr\}\cdot \frac{1}{k}}.
 \end{equation}
 To the first term, we apply Jensen's inequality, which produces
 \[
 \sum_{k=0}^\infty \sqrt{\biggl\{\sum_{0 \leq n \leq k}\left\{ 2^{2(n-k)}\tau_n^4 \right\}\biggr\}\cdot \frac{1}{k}}
 \lesssim
 \sum_{k=0}^\infty 
 \sqrt{\frac{1}{k}}
 \cdot
 \biggl\{\sum_{0 \leq n \leq k}\left\{ 2^{2(n-k)}\tau_n^2 \right\}\biggr\}
 \lesssim
 \sum_{n=0}^\infty \tau_n^2,
 \]
 where the second sum follows from changing the order of summation.
 To the second term in \eqref{eq:g23}, we again apply Cauchy condensation
 \[
 \sum_{k=0}^\infty \sqrt{
 	\sum_{n \geq k}\tau_n^4 
 	\cdot 
 	\frac{1}{k}}
 \lesssim
 \sum_{k=0}^\infty \sqrt{\biggl\{\sum_{j \geq k}\tau_{2^j}^4 \cdot 2^j \biggr\}} \cdot 2^{k/2}
 \lesssim
 \sum_{k=0}^\infty \biggl\{\sum_{j \geq k}\tau_{2^j}^2 \cdot 2^{j/2} \biggr\} \cdot 2^{k/2}
 \lesssim
 \sum_{j=0}^\infty \biggl\{\tau_{2^j}^2 \cdot 2^{j} \biggr\},
 \]
 where the penultimate inequality follows from subadditivity of the $\sqrt{\cdot}$ and the final inequality follows by changing order of summation.  From another application Cauchy condensation, \eqref{eq:g1g2} follows.
  \end{proof}

We remark that sequences for which $\sum_{k=0}^\infty \tau_k^2=\infty$ but are square summable necessarily have some amount of lacunary behavior:
\begin{lemma}
  Suppose $\sum_{k=0}^\infty \tau_k^2 = \infty$ but $\sum_{n=0}^\infty \sigma_n^2 < \infty$, then for any $C>1$ there is a sequence $\left\{ j_k \right\}$ tending to infinity with $j_{k+1}/j_k > C$ for all $k$ so that
  \[
    \sum_{k=1}^\infty \sigma_{j_k}^2 \cdot j_k = \infty.
  \]
  \label{lem:lac}
\end{lemma}
\begin{proof}
  Using Cauchy condensation, we have that for any $m \in \N$ with $m > 1$
  \[
    \sum_{j=1}^\infty \tau_{m^j}^2 \cdot m^j
    =
    \infty
    =
    \sum_{k=0}^\infty \tau_k^2.
  \]
  Let $\left\{ j_k^* \right\}$ be the subsequence of $\{m^j\}$ at which
  \(
  \tau_{m^j} > \tau_{m^{j+1}}.
  \)
  Picking $j_k$ as an $\ell$ in $[j_k^*,m j_k^*)$ that maximizes $\sigma_\ell^2$ produces the desired result, after possibly passing to the subsequence $\left\{ j_{2k} \right\}$ or $\left\{ j_{2k+1} \right\}$.
\end{proof}

\section{Exceptional GAFs}
\label{sec:exceptional}
In this section, we construct GAFs with exceptional properties.  In particular we show the strict inclusions in \eqref{eq:strict}.

\subsection{ $\text{H}^2$--Bloch GAFs are not always BMO GAFs }
Both lacunary and regularly varying $\text{H}^2$--GAFs are VMOA. 
\cite[Theorem 3.5]{sledd1981random} constructs an example of an $\text{H}^2$ random series that is not Bloch, and so not BMOA.
This leaves open the possibility that once an $\text{H}^2$--GAF is Bloch, it additionally is BMO.
We give an example that shows there are $\text{H}^2$--GAFs that are Bloch but not BMO.

Recall \eqref{eq:MI} that for an interval $I \subset [0,1]$ any $p \geq 1$ and any $\text{L}^p(\T)$ function $p,$  
\begin{equation*}
  M_{I}^p(f) \coloneqq \fint_{I}  | f(e(\theta)) - \textstyle{\fint_I }f|^p\,d\theta,
  \qquad
  \text{where}
  \quad
  \fint_I f(x)\,dx \coloneqq
  \frac{1}{|I|} \int_{I} f(x)\,dx.
\end{equation*}

\begin{lemma}
For every $R>0$, there exists $n_0=n_0(R)$ such that for any $n > n_0$ there is a polynomial
\(
f(z)\coloneqq\sum_{k=n}^{m}a_{k}\xi_{k}z^{k},
\)
with the following properties:
\begin{enumerate}[\textnormal{(}i\textnormal{)}]
	\item $\sum_k a_{k}^{2}=1.$
	\item $\Exp \starnorm[f] \geq R$
	\item  \(
	\Exp \Blochnorm[f] \leq C,
	\)
	where $C > 0$ is an absolute constant.
\end{enumerate}
\label{lem:gady}
\end{lemma}

We can then use this lemma to construct the desired GAF.
\begin{theorem}
  There exists an $\text{H}^2$, Bloch, non--BMOA GAF.
  \label{thm:gady}
\end{theorem}
\begin{proof}
  Let $\left\{ \beta_i \right\}$ and $\left\{ R_i \right\}$ be two positive sequences with $\left\{ \beta_i \right\} \in \ell_1$  and $\beta_i R_i \to \infty.$
  Let $f_i$ be a sequence of independent polynomials given by Lemma \ref{lem:gady} having
  \[
    \Exp \starnorm[f] \geq R_i
    \quad
    \text{and}
    \quad
    \Exp \Blochnorm[f] \leq C.
  \]
  The function $f = \sum_i \beta_i f_i$ satisfies for all $\theta \in [0,1]$
  \[
    \Exp |f(e(\theta))|^2 = \sum_{i=1}^\infty \beta_i^2 < \infty,
  \]
  and so $f$ is in $\text{L}^2.$
  The Bloch norm satisfies
  \[
    \Exp \Blochnorm[f] \leq \sum_{i=1}^\infty \Exp \beta_i\Blochnorm[f_i] < \infty.
  \]
  Finally, by the contraction principle (Proposition \ref{prop:contraction})
  \[
    \Exp \starnorm[f] \geq \beta_i \Exp \starnorm[f_i] \geq \beta_i R_i \to \infty,
  \]
  as $i \to \infty.$  Therefore $\starnorm[f] = \infty$ a.s.\ by Corollary \ref{cor:expisenough}.
\end{proof}

\subsection*{Proof of Lemma \ref{lem:gady}}

\subsubsection*{Construction of $f$}
Let $r\in\mathbb{N}$ be some parameter to be fixed later (sufficiently
large). Let $\{\lambda_{i,j}:i,j\in \{ 1,2, \dots, r\} \} \cup \{1\}$ be real numbers that are
 independent over the rationals and that satisfy
 \begin{equation}\label{eq:lambdarange}
\lambda_{i,j}\in[2^{i},2^{i}+4^{-r}].
\end{equation}
Independence gives that for every $\omega\in\{0,\frac{1}{2}\}^{r \times r}$
there is an $m=m(\omega)$ such that 
\begin{equation}\label{eq:mdef}
  |\left\{ m\lambda_{i,j}\right\} -\omega_{i,j}|\le 4^{-r}\qquad \text{ for all } \quad i,j=1,\dotsc,r,
\end{equation}
where as usual $\{x\}=x-\lfloor x\rfloor$ is the fractional value. 

Let $n_0 = 4^r(\max_{\omega} m(\omega) + 1),$
and let $n > n_0$ be arbitrary.
Define
\[
a_{k}=\begin{cases}
\frac{1}{r} & k=\lfloor n\lambda_{i,j}\rfloor\text{ for some }i,j=1,\dotsc,r\\
0 & \text{otherwise.}
\end{cases}
\]
Denote for short $\zeta_{i,j}=\xi_{\lfloor n\lambda_{i,j}\rfloor}$ for
$i,j=1,\dotsc,r$ so that 
\begin{equation}\label{eq:gadyf}
f(z)=\frac{1}{r}\sum_{i,j=1}^{r}\zeta_{i,j}z^{\lfloor n\lambda_{i,j}\rfloor}.
\end{equation}

\subsubsection*{Lower bound for $\Exp\starnorm[f]$}

Define a random variable $\omega\in\{0,\frac{1}{2}\}^{r \times r}$ by
\[
\omega_{i,j}=\begin{cases}
0 & \Re \zeta_{i,j}\ge0\\
\tfrac{1}{2} & \Re \zeta_{i,j}<0.
\end{cases}
\]
Let $I$ be the interval of length $\frac{1}{n}$ centered around
$\frac{m(\omega)}{n}$. 

We will show that $\Exp M_{I}^2(f)$ is large.
To do so, we give an effective approximation for $\Re f$ on $I.$

Define
\[
  g(\theta) \coloneqq \sum_{i=1}^r \Xi_i \cos(2\pi \cdot 2^i n \theta)
  \quad
  \text{where}
  \quad
  \Xi_i \coloneqq \frac{1}{r} \sum_{j=1}^r |\Re \zeta_{i,j} |.
\]
Notice that $g$ is $1/n$--periodic and therefore
\[
  M_{I}^2(g) = %
  \fint_I |g(\theta)|^2\,d\theta = \frac{1}{2}\sum_{i=1}^r \Xi_i^2.
\]
Hence $\Exp m_I^2(g) \geq Cr$ for some absolute constant $C > 0,$ and so it remains to approximate $f$ by $g.$

\begin{claim}   There is a sine trigonometric polynomial $h$
  such that with
  \[
    E = E(\theta) \coloneqq \Re f( e( m(\omega)/n + \theta)) - g(\theta) - h(n\theta),
  \]
  for $|\theta| \leq \tfrac1n,$
  \[ |E(\theta)| \leq 3 \cdot 4^{-r} \cdot \sum_{i,j} |\zeta_{i,j}|. \]
\end{claim}
\begin{proof}
  By \eqref{eq:mdef},
  \[
    d(\tfrac{m(\omega)}{n}\lfloor n \lambda_{i,j} \rfloor - \omega_{i,j}, \Z)
    \leq 4^{-r} + \tfrac{m(\omega)}{n}
    \leq 2\cdot 4^{-r}.
  \]
  By \eqref{eq:lambdarange}, $\lfloor n\lambda_{i,j} \rfloor \in [n2^{i}, n2^{i} + n4^{-r}],$
  and so for $|\theta| \leq \tfrac 1n,$
  \[
    |\theta \lfloor n \lambda_{i,j} \rfloor - \theta n2^{i}| \leq 4^{-r}.
  \]
  Combining these two estimates, for $|\theta| \leq \tfrac 1n,$ for all $i,j=1,2,\dots,r$
  \[
    \left|
    \Re \left( 
    \zeta_{i,j} e( (\tfrac{m(\omega)}{n} + \theta )\lfloor n \lambda_{i,j} \rfloor)
    \right)
    -
    \Re \left( \zeta_{i,j}e(\omega_{i,j} + 2^{i} n \theta) \right)
    \right|
    \leq 3\cdot 4^{-r} |\zeta_{i,j}|.
  \]
  Using that %
  $e(\theta + \tfrac 12) = - e(\theta),$ the claim follows by applying the previous estimate term-by-term to \eqref{eq:gadyf}.
\end{proof}

We now bound the oscillation $M_I^2(f)$ as follows.
Using that $\fint_I g = \fint_I h = 0$ and the orthogonality of $g$ and $h$ on $I$,
\[
  \begin{aligned}
  \biggl[\fint_I | f(e(\theta)) - {\textstyle \fint_I } f|^2\,d\theta\biggr]^{1/2}
  &\geq 
  \biggl[\fint_I | \Re f(e(\theta)) - {\textstyle \fint_I } \Re f|^2 \,d\theta\biggr]^{1/2} \\
  &\geq 
  \biggl[\fint_I | g(\theta)+h(\theta)|^2 \,d\theta\biggr]^{1/2}
  -
  \biggl[\fint_I | E(\theta)|^2 \,d\theta\biggr]^{1/2} \\
  &\geq
  \biggl[\fint_I | g(\theta)|^2 \,d\theta\biggr]^{1/2}
  -
  3\cdot 4^{-r} \cdot \sum_{i,j} |\zeta_{i,j}|\\
  &\geq
  \biggl[ \frac{1}{2}\sum_{i=1}^r \Xi_i^2 \biggr]^{1/2}
  -
  3\cdot 4^{-r} \cdot \sum_{i,j} |\zeta_{i,j}|\\
\end{aligned}
\]
Using \cite[Proposition 4.1]{Girela}, there is a constant $C_2 \geq 1$ so that
\[
  \Exp \starnorm[f]
  \geq
  \frac{1}{C_2}
  \Exp
  \biggl(
  \biggl[\fint_I | f(e(\theta)) - {\textstyle \fint_I } f|^2\,d\theta\biggr]^{1/2}
  \biggr)
  \geq
  \sqrt{r}/C -  C\cdot 4^{-r} r^2,
\]
for some sufficiently large absolute constant $C>0.$

\subsubsection*{Upper bound for $\Exp \Blochnorm[f]$}

We begin by computing 
\[
  f'(z) = \frac{1}{r} \sum_{i,j} \zeta_{i,j} \lfloor n \lambda_{i,j} \rfloor z^{\lfloor n\lambda_{i,j} \rfloor - 1},
\]
with the sum over all $i,j$ in $\{1,2,\dots,r\}.$
Let $\Theta_i = \frac{1}{r}\sum_{j} |\zeta_{i,j}|.$  Then for $t = |z| < 1$
\[
  \begin{aligned}
    (1-|z|)
    |f'(z)|
    &\leq 
    \frac{1-t}{r}\sum_{i,j} |\zeta_{i,j}| n 2^{i+1} t^{n 2^{i} - 1} \\
    &\leq (\max_i \Theta_i) (1-t) \sum_{i} n 2^{i+1} t^{n 2^{i} - 1} \\
    &\leq 2(\max_i \Theta_i) (1-t) \sum_{i=1}^\infty t^{i-1} \leq 2(\max_i \Theta_i).
 \end{aligned}
\]
Using Bernstein's inequality (e.g.  \cite[Theore 2.8.1]{Vershynin}) or standard manipulations (c.f.\ \cite[Exercise 2.5.10]{Vershynin}) for sums of independent random variables, one can check that
\[
\sup_{r \in \bbn} \left[\Exp \max_{i=1,2,\dots,r} \Theta_i \right]  < \infty.
\]
This concludes the proof of Lemma \ref{lem:gady}. \hfill \wasylozenge

\subsection{ BMO GAFs are not always VMO GAFs }

We answer a question of Sledd \cite{sledd1981random}, showing that there are GAFs which are in BMOA but not in VMOA.
We begin by defining a new seminorm on BMOA
\[
  \|f\|_{*,n} \coloneqq \sup_{I : 2^{-n} \leq |I| \leq 2^{-(n-1)}} M_I^1(f),
\]
where the supremum is over intervals $I \subset \R/\Z.$
\begin{lemma}
  There is a constant $c>0$ and an $m \in \N$ so that for all integers $n \geq m$
  and for all polynomials $p$ with coefficients supported in $[2^{n},2^{n+1}]$
    \[
      \starnorm[p] \geq \|p\|_{*,n-m} \geq c\| p \|_{\infty} \geq {c}\starnorm[p].
    \]
  \label{lem:tao}
\end{lemma}
\vspace{-2em}
\begin{proof}
  The first inequality is trivial.  The last inequality is \cite[Proposition 2.1]{Girela}.  Thus it only remains to prove the second inequality.
  Recall $T_n,$ the dyadic trapezoidal kernel \eqref{eq:trapezoid},
  which satisfies for all $n \in \N$ that
  $\widehat T_n(j) = 1$ for $j \in [2^n, 2^{n+1}],$ $\widehat T_n(0) = 0,$ and $\|T_n\|_\infty \leq C\cdot 2^n$ (c.f.\ \eqref{eq:fejer} and \eqref{eq:trapezoid}) for some absolute constant $C>0.$
  From the condition on the support of the coefficients, $p \star T_n = p.$  As the constant coefficient of $T_n$ vanishes, $1 \star T_n =0,$ and therefore we have the identity that for any $I \subset \R/\Z$ and any $\phi \in [0,1]$
  \begin{equation}
    p(e(\phi)) = ((p - {\textstyle \fint_I } p) \star T_n)(z) = \int_0^1 (p(e(\theta))- {\textstyle \fint_I } p)T_n(e(\phi-\theta))\,d\theta.
    \label{eq:pid}
  \end{equation}
  Fix $m \in \N.$  Let $I$ be the interval around $\phi$ of length $2\cdot 2^{m-n}.$  Then for $n \geq m+1,$
  \begin{equation}\label{eq:pest}
    \begin{aligned}
    |p(e(\phi))|
    &\leq 
    \int_{I \cup I^c} |p(e(\theta))- {\textstyle \fint_I } p| \cdot |T_n(e(\phi-\theta))|\,d\theta \\
    &\leq
    \|T_n\|_\infty \int_{I} |p(e(\theta))- {{\textstyle \fint_I }} p|\,d\theta
    +2\|p\|_\infty \int_{I^c} |T_n(e(\phi-\theta))|\,d\theta.
    \end{aligned}
  \end{equation}
  The first summand we control as follows
    \begin{equation}\label{eq:Iest}
    \|T_n\|_\infty \int_{I} |p(e(\theta))- {{\textstyle \fint_I }} p|\,d\theta
    \leq
    6 \cdot 2^{m+1} \fint_{I} |p(e(\theta))- {{\textstyle \fint_I }} p|\,d\theta
    \leq 6 \cdot 2^{m+1} \|p\|_{*,n-m-1}.
  \end{equation}
  For the second summand, using \eqref{eq:trapezoidbound}
  \begin{equation}\label{eq:Icest}
    \begin{aligned}
    2\|p\|_\infty \int_{I^c} |T_n(e(\phi-\theta))|\,d\theta
    &\leq 
    4\|p\|_\infty \int_{2^{m-n}}^{1/2} |T_n(e(\theta))|\,d\theta \\
    &\leq 
    80 \cdot 2^{-n} \|p\|_\infty \int_{2^{m-n}}^{1/2} |1-e(\theta)|^{-2}\,d\theta \\
    &\leq 
    20 \cdot 2^{-n} \|p\|_\infty \int_{2^{m-n}}^{\infty} \theta^{-2}\,d\theta \\
    &\leq 
    20 \cdot 2^{-m} \|p\|_\infty.
  \end{aligned}
  \end{equation}
  Applying \eqref{eq:Iest} and \eqref{eq:Icest} to \eqref{eq:pest},
  \[
    \|p\|_\infty 
    \leq 6 \cdot 2^{m+1} \|p\|_{*,n-m-1}
    +20 \cdot 2^{-m} \|p\|_\infty.
  \]
  Taking $m=5,$ we conclude
  \[
    \|p\|_\infty \leq 2^{10} \|p\|_{*,n-6}.
  \]
\end{proof}

The previous lemma allows us to estimate $\starnorm$ of polynomials supported on dyadic blocks efficiently in terms of the supremum norm.  Hence, we record the following simple observation.
\begin{lemma}\label{lem:blockinfinity}
  For any $n \geq 2,$ let $f_n$ be the Gaussian polynomial
  \[
    f_n(z) = \frac{1}{\sqrt{n\log n}} \sum_{k=n}^{2n-1} \xi_k z^k.
  \]
  Then there is an absolute constant $C>0$ so that
  \[
    C^{-1} < \Exp\|f_n\|_{\infty} < C.
  \]
  And for all $t \geq 0,$
  \[
    \Pr\left[ |\|f_n\|_\infty - \Exp\|f_n\|_{\infty}| > t  \right] \leq 2e^{-(\log n)t^2}.
  \]
\end{lemma}
\begin{proof}
  Observe that the family $\left\{ f_n(e(k/n)) : 0 \leq k < n \right\}$ are i.i.d.\ complex Gaussians of variance $\frac{1}{\log n}.$  Hence 
  \[
    \Exp \|f_n \|_\infty \geq \Exp \max_{0 \leq k < n} |f_n(e(k/n))| \geq C.
  \]
  for some constant $C>0$ (see \cite[Exercise 2.5.11]{Vershynin}).  Conversely, there is an absolute constant so that for any polynomial $p$ of degree $2n$ (e.g.\ see \cite{Rakhmanov})
  \[
    \|p \|_\infty \leq C \max_{0 \leq k \leq 4n} |p(e(k/(4n)))|.
  \]
  Hence using that each of $f_n(e(k/(4n)))$ is complex Gaussian of variance $\frac{1}{\log n}$, we conclude that there is another constant $C>0$ so that
  \[
    \Exp \|f_n \|_\infty 
    \leq C
  \]
(see \cite[Exercise 2.5.10]{Vershynin}).
  The concentration is a direct consequence of Proposition \ref{prop:glip}.
\end{proof}

Let $\left\{ n_k \right\}$ be a monotonically increasing sequence of positive integers, to be chosen later.  Let $\left\{ f_{k} \right\}$ be independent Gaussian polynomials as in Lemma \ref{lem:blockinfinity} with $n=2^{n_k}.$ %
Let $\left\{ a_k \right\}$ be a non--negative sequence.
Define $g = \sum_{k=1}^\infty a_k f_k.$  Under the condition that $\sum_{k=1}^\infty \frac{a_k^2}{n_k} < \infty,$ $g$ is an $\text{H}^2$--GAF.

\begin{lemma}
  Let $n_1=1$ and define $n_{k+1}=3^{n_k}$ for all $k \geq 0.$
  If the sequence $\left\{ a_k \right\}$ is bounded, then $g$ is in BMOA almost surely.
  If furthermore $\lim_{k \to \infty} a_k = 0,$ then $g$ is in VMOA almost surely.
  \label{lem:gbmoa}
\end{lemma}
\begin{proof}
  Without loss of generality we may assume all $a_k \leq 1.$
  Observe that
  \[
    \starnorm[g] = \sup_{\ell \in \N} \|g\|_{*,\ell}.
  \]
  Therefore if $\sup_{\ell \in \N} \|g\|_{*,\ell} < \infty \As,$ then $g$ is in BMOA. 
  If furthermore
  \( \lim_{\ell \to \infty} \|g\|_{*,\ell} =0 \As,
  \)
  then $g$ is in VMOA almost surely.

Put $g_j = a_j f_j$ for all $j \in \bbn$.  Fix $\ell \in \N$ and let $k$ be such that $n_{k-1} - m \leq \ell \leq n_k - m,$ where $m$ is the constant from Lemma \ref{lem:tao}, and decompose $g = g_{<k-1} + g_{k-1} + g_k + g_{>k}.$
  Then
  \[
    \|g_{>k}\|_{*,\ell} \leq 2^{\ell/2} \| g_{>k}\|_{2} \leq 2^{n_k/2} \| g_{>k}\|_{2},
  \]
  which follows from Cauchy--Schwarz applied to $M_{I}^1(g_{>k})$ for an interval $|I| \geq 2^{-\ell}.$
  On the other hand
  \[
    \|g_{<k}\|_{*,\ell} 
    \leq 
    2^{-\ell+1} \|g'_{<k}\|_\infty 
    \leq  
    2^{-\ell+1} 2^{n_{k-2}+1} \|g_{<k}\|_\infty 
    \leq 2^{-n_{k-1} + n_{k-2} + m +2} \|g_{<k}\|_\infty,
  \]
  where the penultimate inequality is Bernstein's inequality for polynomials.
  We conclude that
  \[
    \|g\|_{*,\ell}
    \leq  
    2^{-n_{k-1} + n_{k-2} + m +2} \|g_{<k}\|_\infty
    +2 \|g_{k}\|_\infty
    +2 \|g_{k-1}\|_\infty
    +2^{n_k/2} \| g_{>k}\|_{2}.
  \]

  Using Lemma \ref{lem:blockinfinity} and Borel-Cantelli,
  \[
    D \coloneqq \sup_{k} \|f_k\|_\infty < \infty \quad \As
  \]
  In particular
  \[
    \|g_{<k}\|_\infty \leq k D.
  \]
  Meanwhile,
  each $\left\{ \|f_j\|_2^2 \cdot 2 \cdot 2^{n_j} \log(2^{n_j}) \right\}$ is distributed as independent $\chi^2$ random variable with $2^{n_j+1}$ degrees of freedom.  Hence
  \[
    R \coloneqq \sup_{j} \{ \|g_j\|_2 \sqrt{n_j}\} < \infty \quad \As
  \]
Therefore,
  \[
    \|g_{>k}\|_2^2 = \sum_{j > k} \|g_j\|_2^2 \leq R^2 \sum_{j > k} \frac{1}{n_j} \leq \frac{3 R^2}{n_{k+1}}.
  \]
  Finally, we have
  \[
    \|g\|_{*,\ell}
    \leq  
    2^{-n_{k-1} + n_{k-2} + m +2}kD
    +2(a_{k-1} + a_k)D
    +\sqrt{3} \cdot 2^{n_k/2}\frac{R}{\sqrt{n_{k+1}}}.
  \]
  By our choice of $n_k$ (recalling $k=k(\ell)$),  the last expression is \emph{uniformly} bounded in $\ell$ almost surely.
  In addition,  if $a_k \to 0$, then
  \[
    \limsup_{\ell \to \infty}
    \|g\|_{*,\ell}
    \leq 
    \limsup_{k \to \infty}
    2(a_{k-1} + a_k)D
    =0.
  \]
\end{proof}

\begin{remark}
  A more careful analysis of $\|f_{>k}\|_{*,\ell}$ reveals that it suffices to have $n_{k+1}/n_k > c > 1$ for some $c$ to bound $\|f_{>k}\|_{*,\ell}$ uniformly over all $\ell$. We will not pursue this here.
\end{remark}

We now turn to proving the existence of the desired GAF.
\begin{theorem}
There is a $\text{BMO}$ GAF which is almost surely not a $\text{VMO}$ GAF.
  \label{thm:BMOVMO}
\end{theorem}
\begin{proof}
We let $g$ be as in Lemma \ref{lem:gbmoa} with $a_k=1$ for all $k.$   
By making $t$ sufficiently small and using the contraction principle (Proposition \ref{prop:contraction}), Lemma \ref{lem:tao} and Lemma \ref{lem:blockinfinity}, for all $k \in \N$
  \[
    2
    \Pr\left( 
    \|g\|_{*,n_k - m} > t
    \right)
    \geq
    \Pr\left( 
    \|f_{k}\|_{*,n_k - m} > t    \right)
    \geq
    \Pr\left( 
    \|f_{k}\|_{\infty} > ct
    \right)
    \geq \tfrac{1}{2}.
  \]
  Therefore by the reverse Fatou Lemma,
  \[
  \Pr\left( 
  \limsup_{k \to \infty} \|g\|_{*,n_k - m} > t
  \right)
  \geq
  \limsup_{k \to \infty} 
  \Pr\left( 
  \|g\|_{*,n_k - m} > t
  \right)
  \geq \tfrac{1}{4}.
  \]
  This implies, by Proposition \ref{prop:01}, that $g$ is not in VMOA a.s.

\end{proof}

Finally, we show there is a VMO GAF which is not Sledd.

\begin{lemma}
  There is an absolute constant $c>0$ so that for all $\epsilon > 0$ there is an $n_0(\epsilon)$ sufficiently large so that for all $n \geq n_0$ and for all intervals $I \subset [0,1]$ with $|I|=\epsilon$
  \[
    \Pr\left( \exists~J\subset I \text{ an interval with } |J|=\tfrac{c}{n} \text{ such that } \min_{x \in J} |f_n(x)| > \tfrac{1}{4} \right) \geq \tfrac{1}{3},
  \]
  where $f_n$ is as in Lemma \ref{lem:blockinfinity}.
  \label{lem:nesting}
\end{lemma}
\begin{proof}
  We again use the 
  observation that the family $\left\{ f_n(e(k/n)) : 0 \leq k < n \right\}$ are i.i.d.\ complex Gaussians of variance $\frac{1}{\log n}.$
  Let $I$ be an interval as in the statement of the lemma.  Let $I'$ be the middle third of that interval.  Then for any $n,$ there are at least $n\epsilon/4$ many $k$ so that $k/n$ are in $I'.$  For any such $k,$ and any $t$
  \[
    \Pr[ |f_n(e(k/n))| > t ] = e^{-(\log n)t^2}.
  \]
  Hence if we define $n_0$ so that $n_0^{3/4}\epsilon = 4\log(3)$ then for all $n \geq n_0,$
  \[
    \Pr\left[ \forall ~ k : k/n \in I', |f_n(e(k/n))| \leq\tfrac{1}{2}  \right] \leq e^{ - \tfrac14 n^{3/4}\epsilon} \leq \tfrac{1}{3}.
  \]

  Using Bernstein's inequality, and Lemma \ref{lem:blockinfinity} there is an absolute constant so that
  \[
    \|f_n'\|_\infty \leq 2n \|f_n\|_\infty \leq Cn,
  \]
  except with probability $\frac{1}{n}.$
  Hence, if we let $J$ be the interval of length $c/n$ around a point in $I'$ where $|f_n(e(k/n))| > \tfrac{1}{2}$,
  then $\min_{x \in J} |f_n(x)| \geq \frac14$ except with probability $\frac{2}{3}.$
\end{proof}

\begin{theorem}
  There exists a GAF that is almost surely in VMOA and which is almost surely not Sledd.
  \label{thm:VMOAnotSledd}
\end{theorem}
\begin{proof}

  We let $g$ be as in Lemma \ref{lem:gbmoa} with $a_k \to 0$ to be defined, so that $g$ is almost surely in VMOA.

  We define a nested sequence of random intervals $\left\{ J_\ell \right\}.$ Let $J_0=[0,1].$
  Define a subsequence $n_{k_\ell}$ inductively by letting $n_{k_\ell}$ be the smallest integer bigger than $n_0( \tfrac{c}{n_{k_{\ell-1}}})$ for $\ell > 1$ and $n_0(c)$ for $\ell = 1.$  
  Let $a_{n_{k_\ell}} = \frac{1}{\sqrt{\ell}},$ and let $a_j=0$ if $j$ is not in $\left\{ n_{k_\ell} \right\}.$  

  We say that an interval $J_\ell$ \emph{succeeds} if there is a subinterval $J'$ of length $\tfrac{c}{n_{k_{\ell-1}}}$ such that $\min_{x \in J'} |f_{k_\ell}| > \tfrac{1}{4}.$  If the interval $J_\ell$ succeeds, we let $J_{\ell+1} = J',$ and otherwise we let $J_{\ell+1}$ be the interval of length $\tfrac{c}{n_{k_{\ell-1}}}$ that shares a left endpoint with $J_\ell.$  
  The nested intervals $\overline{J_\ell}$ decrease to a point $x,$ and   
  \[
    \RSledd[f]^2 
    \geq \sum_{\ell=1}^\infty \frac{1}{\ell}|f_{k_\ell}(x)|^2
    \geq \sum_{\ell=1}^\infty \frac{1}{16\ell}\one[ J_\ell \text{ succeeds}].
  \]
  From Lemma \ref{lem:nesting}, the family $\left\{ \one[ J_\ell \text{ succeeds}] \right\}$ are independent Bernoulli with parameter at least $\frac{1}{3}.$ Then by \cite[Chapter 3, Theorem 6]{Kahane}, this series is almost surely infinite.
\end{proof}

\printbibliography[heading=bibliography]%

\end{document}